\documentclass[11pt]{amsart}

\usepackage{amsmath}
\usepackage{amssymb}
\usepackage{graphicx}
\usepackage{color}
\usepackage{hyperref}

\newtheorem{theorem}{Theorem}[section]
\newtheorem{corollary}[theorem]{Corollary}
\newtheorem{lemma}[theorem]{Lemma}

\theoremstyle{definition}
\newtheorem{definition}[theorem]{Definition}

\def\gJt{\gJ(\bth_t)}
\def\gJu{\gJ(\ubold_t)}
\def\gJw{\gJ(\w_t)}

\def\xbb{\bar{\x}}
\def\bD{{\boldsymbol \D}}

\def\mub{\bar{\mu}}
\def\ebar{\bar{\e}}

\def\oneb{{\bf 1}}

\def\c{{\bf c}}

\def\eb{{\bf e}}

\def\h{{\bf h}}

\def\ubold{{\bf u}}
\def\vbold{{\bf v}}
\def\w{{\bf w}}
\def\x{{\bf x}}

\def\z{{\bf z}}

\def\B{{\mathcal B}}
\def\C{{\mathcal C}}

\def\F{{\mathcal F}}

\def\M{{\mathcal M}}

\def\R{{\mathbb R}}

\def\al{\alpha}
\def\de{\delta}
\def\delb{\bar{\delta}}
\def\D{\Delta}
\def\e{\epsilon}
\def\g{\gamma}

\def\la{\lambda}
\def\LA{\Lambda}
\def\om{\omega}
\def\OM{\Omega}

\def\s{\sigma}
\def\SI{\Sigma}

\def\Th{\Theta}

\def\bD{{\boldsymbol \D}}

\def\bth{{\boldsymbol \theta}}
\def\bphi{{\boldsymbol \phi}}

\def\bzeta{{\boldsymbol \zeta}}

\def\ab{\bar{a}}

\def\bb{\bar{b}}
\def\bul{\underline{b}}

\def\dh{\hat{d}}

\def\kb{\bar{k}}

\def\tai{t \ap \infty}

\def\ap{\rightarrow}

\def\seq{\subseteq}

\def\bi{\{0,1\}}

\def\bp{\{-1,1\}}
\def\bz{{\bf 0}}
\def\imp{\; \Longrightarrow \;}

\def\fa{\; \forall}

\def\half{\frac{1}{2}}

\def\as{\mbox{ a.s.}}

\def\nm{\Vert}

\renewcommand{\and}{\mbox{$\wedge$}}

\def\bzt{\bzeta_{t+1}}

\def\gJ{\nabla J}

\newcommand{\bc}{\begin{center}}
\newcommand{\ec}{\end{center}}
\newcommand{\be}{\begin{equation}}
\newcommand{\ee}{\end{equation}}
\newcommand{\bd}{\begin{displaymath}}
\newcommand{\ed}{\end{displaymath}}
\newcommand{\ba}{\begin{array}}
\newcommand{\ea}{\end{array}}
\newcommand{\ben}{\begin{enumerate}}
\newcommand{\een}{\end{enumerate}}
\newcommand{\bit}{\begin{itemize}}
\newcommand{\eit}{\end{itemize}}
\newcommand{\beq}{\begin{eqnarray}}
\newcommand{\eeq}{\end{eqnarray}}
\newcommand{\btab}{\begin{tabular}}
\newcommand{\etab}{\end{tabular}}
\newcommand{\bfig}{\begin{figure}}
\newcommand{\efig}{\end{figure}}
\newcommand{\btp}{\begin{tikzpicture}}
\newcommand{\etp}{\end{tikzpicture}}

\newcommand{\nmeu}[1]{ \nm #1 \nm_2 }
\newcommand{\nmeusq}[1]{ \nm #1 \nm_2^2 }

\newcommand{\IP}[2]{ \langle #1 , #2 \rangle }

\def\nmsl1{\nm_{{\rm SL1}}}

\def\gJ{\nabla J}
\def\gJt{\gJ(\bth_t)}

\def\bD{{\boldsymbol \D}}

\numberwithin{equation}{section}

\title[Momentum-Based Algorithms with Time-Varying Parameters]
{Convergence of Momentum-Based Optimization Algorithms \\
with Time-Varying Parameters}

\author[M. Vidyasagar]{Mathukumalli Vidyasagar}

\address[M. Vidyasagar]{Indian Institute of Technology Hyderabad,
Kandi, Telangana 502285, India}
\email{\tt m.vidyasagar@iith.ac.in}



\begin{document}

\begin{abstract}
In this paper, we present a unified algorithm for stochastic optimization
that makes use of a ``momentum'' term; in other words, the stochastic
gradient depends not only on the current true gradient of the objective
function, but also on the true gradient at the previous iteration.
Our formulation includes the Stochastic Heavy Ball (SHB) and the
Stochastic Nesterov Accelerated Gradient (SNAG) algorithms as
special cases.
In addition, in our formulation, the momentum term is allowed to vary
as a function of time (i.e., the iteration counter).
The assumptions on the stochastic gradient are the most general in the
literature, in that it can be biased, and have a conditional variance
that grows in an unbounded fashion as a function of time.
This last feature is crucial in order to make the theory applicable
to ``zero-order'' methods, where the gradient is estimated using
just two function evaluations.

We present a set of sufficient conditions for the convergence of the
unified algorithm.
These conditions are natural generalizations of the familiar Robbins-Monro
and Kiefer-Wolfowitz-Blum conditions for standard stochastic gradient
descent.
We also analyze another method from the literature for the SHB
algorithm with a time-varying momentum parameter, and show that it
is impracticable.

\textbf{This paper is dedicated to Alain Bensoussan on the occasion of
his 85th birthday.}
\end{abstract}

\maketitle


\section{Introduction}\label{sec:Intro}

\subsection{Scope of the Paper}\label{ssec:11}

In this paper, we study ``momentum-based'' algorithms for optimization,
where the update of the current guess is based on not only the
current guess and a random search direction (stochastic gradient), but also
the previous guess.
Familiar examples of momentum-based algorithms are the Stochastic
Heavy Ball (SHB) and the Stochastic Nesterov Accelerated Gradient (SNAG)
algorithms.
Until now, by and large these algorithms have been treated separately
in the literature.
In this paper, we introduce a general model for optimization algorithms
that incorporate a momentum term,
which is versatile enough to include both the SHB and SNAG algorithms
as special cases.
Since Stochastic Gradient Descent (SGD) is a special case of SHB, our
analysis is applicable to SGD as well, though SGD is \textit{not}
momentum-based.
A proper literature review is given in Section \ref{sec:Review}.
But for now we note that
we are aware of only \cite{Liu-et-al-NC23} and its predecessors
\cite{Zou-et-al-arxiv18,Yan-et-al-IJCAI18} that present a ``Stochastic
Unified Momentum (SUM)'' algorithm.
Our formulation is more general than the ones in these references,
in that the momentum
parameter is allowed to vary with time (iteration counter).
More important, we permit all the parameters in the algorithm,
including the step size and the momentum parameter, to be random.
This makes the present analysis applicable to ``block'' updating,
unlike previous work in
\cite{Liu-et-al-NC23,Zou-et-al-arxiv18,Yan-et-al-IJCAI18}.

Our analysis is not restricted to convex objectve functions.
Assuming only that the gradient $\gJ(\cdot)$ of the objective
function $J: \R^d \ap \R$ is globally Lipschitz-continuous,
we show that 
\be\label{eq:111}
\liminf_{\tai} \nmeu{\gJt} = 0 ,
\ee
where $\{ \bth_t \}$ are the iterations of the argument.
This is a stronger conclusion than 
\be\label{eq:112}
\lim_{\tai} \min_{0 \leq \tau \leq t} \nmeusq{\gJ(\bth_\tau)} = 0 .
\ee
To see this, suppose $\gJ(\bth_T) = \bz$ for some $T$;
then \eqref{eq:112} holds, but not necessarily \eqref{eq:111}.
If the convexity of $J(\cdot)$ is replaced by
the Kurdyka-{\L}ojasiewicz (KL) property,
then we deduce the almost-sure convergence of the iterations to the
set of global minimizers.
If the objective function satisfies the stronger Polyak-{\L}ojasiewicz (PL)
property, then we not only deduce the almost-sure convergence of the algorithm,
but also find bounds on the \textit{rate} of convergence.
When the stochastic gradient is unbiased and has bounded variance,
these bounds match the best-achievable upper bounds for arbitrary
\textit{convex} objective functions, as shown in \cite{Arjevani-et-al-MP23}.
However, a novel feature of our analysis is that the assumptions on the
stochastic gradient are the most general (i.e., least restrictive) to date.

We use the same assumptions as in \cite{MV-RLK-SGD-arxiv23,MV-RLK-SGD-JOTA24},
reproduced here in Section \ref{ssec:21}.
Specifically, it permitted for the stochastic gradient to be ``biased''
in the sense that its conditional expectation need not equal the true gradient.
Further, the conditional variance of the stochastic gradient is permitted to
grow with time (i.e., the iteration counter), and also as function
of the distance to the set of optima, in an unbounded manner.
This generality allows us to analyze so-called zeroth-order methods,
which approximate the gradient using only function evaluations.
Existing papers cannot directly handle this approach.
We derive conditions under which the algorithm converges almost surely,
and not just in probability or in the mean.
Our conditions are natural generalizations of the well-known Robbins-Monro
(RM) or or Kiefer-Wolfowitz-Blum (KWB) conditions, and are thus less
restrictive than those in the literature.
All of these points are brought more clearly in Section \ref{sec:Review},
in which
we carry out a literature survey, and compare our results to previous results.

\subsection{Organization of the Paper}\label{ssec:12}

The paper is organized as follows:
In Section \ref{sec:General}, we present the general algorithm that
unifies both SHB and SNAG under a common formalism.
In Section \ref{sec:Review}, we present a literature review, in two parts.
First, we give a broad overview, and then we discuss a handful of papers
that are directly relevant to this paper.
The main results of the paper are presented in Section \ref{sec:Main}.
These main theorems include both the case of  ``full coordinate updating'',
where all components of the argument are updated at each iteration,
and ``block updating,'' where only some components (possibly chosen at
random) are updated at each iteration.
The proofs of the main theorems are given in Section \ref{sec:Proofs}.
The conclusions of the paper and some suggestions for future research
are given in Section \ref{sec:Conc}.
The Appendix contains
some relevant material showing why the approach suggested in
\cite{Sebbouh-et-al-CoLT21} is not feasible when the momentum parameter
varies with $t$.

\section{Formulation of the General Algorithm}\label{sec:General}

In this section, we set up a general class of algorithms studied
in the paper, and then show that both the 
Stochastic Heavy Ball (SHB) and Stochastic Nesterov Accelerated Gradient
(SNAG) algorithms become special cases of the general algorithm, with
suitable choices of the parameters.

The problem is to minimize a $\C^1$ objective function $J: \R^d \ap \R$.
The assumptions on $J(\cdot)$ are stated in Section \ref{ssec:21}; here
we state only the general form of the algorithm and its special cases.

Throughout the paper the symbol $\bth_t \in \R^d$ denotes the
``argument'' of the function $J(\cdot)$, which we hope will converge to
the set of minimizers of $J(\cdot)$.
However, the iterative algorithm is not based on updating $\bth_t$
directly.
Rather, it is defined in terms of two auxiliary vectors, denoted here
by $\w_t$ and $\vbold_t$.
The relationship between $\w_t$ and $\bth_t$ is given by
\be\label{eq:123}
\w_t = \bth_t + \e_t \vbold_t .
\ee
The general algorithm consists of updating formulas for $\w_t$ and $\vbold_t$,
as follows:
\be\label{eq:121}
\w_{t+1} = \w_t + a_t \vbold_t - b_t \al_t \h_{t+1} , 
\ee
\be\label{eq:122}
\vbold_{t+1} = \mu_t \vbold_t - \al_t \h_{t+1} .
\ee
In this formulation, $\al_t$ is known as the \textbf{step size},
also known as the ``learning rate'' in the machine learning literature,
while $\mu_t$ is known as the \textbf{momentum parameter}.
In addition, $\{ a_t \}$, $\{ b_t \}$, $\{ \e_t \}$ are sequences of real
constants that can be adjusted to make \eqref{eq:121}--\eqref{eq:122}
mimic various standard algorithms.
Usually they are viewed differently from the sequences $\{ \mu_t \}$
and $\{ \al_t \}$.
More assumptions on these three constants are given in Section \ref{ssec:21}.
Further, $\h_{t+1}$ is a random vector that is an approximation to
$\gJw$ (note, not necessarily to
$\gJt$), known as the \textbf{stochastic gradient}.
All of our analysis pertains to the behavior of $\w_t$ and $\vbold_t$.
However, the conclusions can be translated back to the behavior of
the original argument variable $\bth_t$, using \eqref{eq:123}.

Now it is shown that both SHB and SNAG are special cases of \eqref{eq:121}
and \eqref{eq:122} for suitable choices of the various constants.
Since Stochastic Gradient Descent (SGD) is a special case of SHB,
it too is a special case of the above algorithm.
However, SGD is not a momentum-based algorithm.

The Heavy Ball method was first introduced in \cite{Polyak-CMMP64},
where both $\al_t$ and $\mu_t$ are fixed constants.
In the present context, the objective is to solve the equation
$\gJ(\bth) = \bz$ using noisy measurements of the gradient.
In this more general setting, the formulation of SHB is
\be\label{eq:124}
\bth_{t+1} = \bth_t + \mu_t ( \bth_t - \bth_{t-1}) - \al_t \h_{t+1} ,
\ee
where $\al_t$ is the step size, $\mu_t$ is the momentum parameter,
and $\h_{t+1}$ is a random approximation to $\gJt$.
To put \eqref{eq:124} in the form \eqref{eq:121}--\eqref{eq:122}, define
\be\label{eq:125}
\vbold_t := \bth_t - \bth_{t-1} , \w_t = \bth_t .
\ee
With these definitions, it is easy to show that the update equations
for $\w_t$ and $\vbold_t$ are
\be\label{eq:126}
\w_{t+1} = \w_t + \mu_t \vbold_t - \al_t \h_{t+1} ,
\ee
\be\label{eq:127}
\vbold_{t+1} = \mu_t \vbold_t - \al_t \h_{t+1} .
\ee
See for example \cite{Sebbouh-et-al-CoLT21}.
These equations are of the form \eqref{eq:121}--\eqref{eq:122} if we define
\bd
\e_t = 0 , a_t = \mu_t , b_t = 1 .
\ed

The Nesterov Accelerated Gradient (NSG) algorithm was introduced in
\cite{Nesterov-Dokl83}.
In the current notation, with the possibility of the gradient being
stochastic, and the momentum coefficient being allowed to vary with $t$,
it can be stated as follows (following \cite[Eqs. (3)--(4)]
{Sutskever-et-al-ICML13}):
\be\label{eq:128}
\bth_{t+1} = \bth_t + \mu_t ( \bth_t - \bth_{t-1} ) - \al_t \h_{t+1} ,
\ee
where $\h_{t+1}$ is a random approximation of
$\gJ( \bth_t + \mu_t ( \bth_t - \bth_{t-1} ))$, and not $\gJt$.
We analyze \eqref{eq:128} using the reformulation in
\cite[Eqs.\ (6)--(7)]{Bengio-et-al-ICASSP13}.
To accommodate the shift in the argument of $\gJ(\cdot)$, we proceed
as follows:
Define 
\be\label{eq:129}
\vbold_t = \bth_t - \bth_{t-1} , \w_t = \bth_t + \mu_t \vbold_t .
\ee
Then the updating formulas are given by 
\cite[Eqs.\ (6)--(7)]{Bengio-et-al-ICASSP13} as
\be\label{eq:1210}
\vbold_{t+1} = \mu_t - \al_t \h_{t+1} ,
\ee
which is the same as \eqref{eq:127}, and
\be\label{eq:1211}
\w_{t+1} = \w_t + \mu_{t+1} \mu_t \vbold_t - (1 + \mu_{t+1}) \al_t \h_{t+1} .
\ee
These equations are of the form \eqref{eq:121} and \eqref{eq:122} with
\be\label{eq:1212}
\e_t = 1 \fa t , a_t = \mu_{t+1} \mu_t , b_t = 1 + \mu_{t+1} .
\ee

Finally, since SGD is a special case of SHB with $\mu_t \equiv 0$ for all $t$,
it too is a special case of the general algorithm
\eqref{eq:121}--\eqref{eq:122}.

\section{Literature Review}\label{sec:Review}

In this section, we present a very brief review of the relevant
results from the literature on SHB and SNAG, and compare the results
in this paper against those.
Because the area of optimization is vast, our comparison is restricted
only to those papers that study the convergence of the SHB or SNAG algorithms.
A more detailed review of momentum-based algorithms is given in 
\cite[Section 1.1]{MV-TUKR-arxiv25}.
We begin with some general remarks on various approaches, and then come to
the papers that are the most relevant to our paper.

\subsection{General References}\label{ssec:31}

The Heavy Ball (HB) algorithm is introduced in \cite{Polyak-CMMP64}
for the minimization of a strictly convex objective function $J(\cdot)$
mapping $\R^d$ into $\R$.
Suppose the condition number of the Hessian $\nabla^2 J(\bth)$
(that is, the ratio of the largest to the smallest eigenvalues)
is bounded by $R$ as a function of $\bth$.
Moreover, let $J^*$ denote the global minimum of $J(\cdot)$.
Then the HB method can be shown to converge at the rate of
$\exp(-2 \sqrt{(R-1)/(R+1)} t)$.
This is faster than the rate of $\exp(-2(R-1)/(R+1) t)$ for
steepest descent (GD).
The Nesterov Accelerated Gradient (NAG) method is introduced in
\cite{Nesterov-Dokl83}.
Subsequent research has shown that (i) for a convex function $J(\cdot)$
with $\gJ(\cdot)$ being Lipschitz-continuous, the iterations converge
to the optimal value at a rate of $O(t^{-2})$, compared to $O(t^{-1})$
for GD.
(ii) Moreover, no method can achieve a faster rate.

After the publication of these two seminal papers, a great deal
of analysis was carried out on these algorithms.
We mention only a small part of the literature that is directly relevant
to the present paper.
A fruitful approach to analyzing the behavior of momentum-based
algorithms such as HB and NAG is to study their
associated ODEs on $\R^d$.
Whereas the ODE associated with GD is of first-order, the ODEs
associated with momentum-based methods are of second-order, due
to the presence of the ``delay'' terms.
The ODE associated with NAG is analyzed in \cite{Su-Boyd-Candes16}, when
the step size $\al$ is held constant, while the momentum coefficient
$\mu_t$ varies with time.
It is shown that the ``optimal'' schedule for $\mu_t$ is
$\mu_t = (t+2)/(t+5)$.
In \cite{Aujol-et-al-SIAMO19}, the rate of convergence of this ODE
is analyzed further by imposing additional structure on $J(\cdot)$,
such as the Kurdyka-{\L}ojasiewicz property\footnote{This property
and the stronger
Polyak-{\L}ojasiewicz property are defined in Section \ref{ssec:21}.}
It is shown that, in certain situations, it is possible for classical
steepest descent method to outperform NAG.
The second-order ODE associated with HB 
is analyzed in \cite{Apido-et-al-JGO22,Aujol-et-al-MP23},
when $J(\cdot)$ satisfies the Polyak-{\L}ojasiewicz property.
In all of the above formulations, it is assumed that the ``stochastic
gradient'' $\h_{t+1}$
equals the true gradient $\gJt$; thus these models do not allow
for measurement errors.
Hence the analysis applies only to HB or NAG, not SHB or SNAG.

Now we come to more recent research on HB.
In much of the literature, attention is focused in the \textit{convergence
in expectation}, or \textit{convrgence in probability} of various algorithms.
In the review paper \cite{Bottou-et-al-SIAM18},
the emphasis is almost exclusively on convergence in expectation.
SHB and SNAG are discussed in \cite[Section 7]{Bottou-et-al-SIAM18}.
Other research on the convergence of HB
(without establishing almost sure convergence)
is summarized very well on page 3 of \cite{Sebbouh-et-al-CoLT21}
and Section 1.1 of \cite{Liu-Yuan-arxiv22}.

In \cite{Ghadimi-et-al-ECC15}, the authors analyze the HB algorithm
where $\h_{t+1} = \gJt$; thus there is no provision for measurement noise,
so that the algorithm being analyzed is HB and not SHB.
The function $J(\cdot)$ is assumed to be convex, and to have a globally
Lipschitz-continuous gradient.
The authors \textit{do not} show that $J(\bth_t)$ converges to the global
minimum of $J(\cdot)$.
Rather, they show that \textit{the average of the first $t$ iterations}
converges to the minimum value of the function $J(\cdot)$.
In \cite{Gadat-et-al-EJS18}, the authors study the SHB for
some classes of nonconvex functions.
It is assumed that the stochastic gradient is unbiased, i.e., that
$E_t(\h_{t+1}) = \gJt$, so that $\x_t = \bz$ for all $t$.
The iterations are shown to converge to a minimum,
but at the cost of ``uniformly elliptic bounds'' on the measurement error
$\bzt$, which are very restrictive.

\subsection{Specific References}\label{ssec:32}

Now we discuss in detail
a few papers that are most closely related to the present paper.
Note that any stochastic algorithm generates \textit{one sample path}
of a stochastic process.
Thus it is very useful to know that the algorithm converges
to the desired limit \textit{almost surely}.
However, there are only a handful of papers that establish the
almost-sure convergence of SHB and/or SNAG.
These are discussed in detail in this subsection.
We also discuss another paper that presents a ``stochastic unified
momentum'' algorithm.

In \cite{Sebbouh-et-al-CoLT21}, the objective function is an expected
value, of the form (\cite[Eq.\ (1)]{Sebbouh-et-al-CoLT21})
\bd
J(\bth) = E_{\w \sim P} F(\bth,\w) .
\ed
The function $F(\cdot,\w)$ is convex for each $\w$, and its gradient
is Lipschitz-continuous with constant $L_\w \leq L$ for all $\w$.
Thus the same holds for $J(\cdot)$ as well.
The stochastic gradient is chosen as (\cite[Eq.\ (SHB)]{Sebbouh-et-al-CoLT21})
\bd
\h_{t+1} = \nabla_{\bth_t} F(\w_{t+1},\bth_t) ,
\ed
where $\w_{t+1}$ is chosen i.i.d.\ with distribution $P$.
Effectively this means that
the stochastic gradient is unbiased.
Also, it is assumed that, for some constant $\s^2$, the conditional variance
$CV_t(\h_{t+1})$ of the stochastic gradient is bounded by
(\cite[Eq.\ (5)]{Sebbouh-et-al-CoLT21})
\bd
CV_t(\h_{t+1}) \leq 4L (J(\bth_t) - J^* ) + \s^2 ,
\ed
where $J^*$ is the infimum of $J(\cdot)$.
In \cite{Sebbouh-et-al-CoLT21}
the authors study the SHB with time-varying parameter $\mu_t$, namely
\be\label{eq:313}
\bth_{t+1} = \bth_t - \al_t \h_{t+1} + \mu_t (\bth_t - \bth_{t-1}) ,
\ee
It is suggested how to convert \eqref{eq:313} above into
\textit{two} equations, which do not contain any ``delayed'' terms.
Specifically, the authors iteratively define
\be\label{eq:314}
\la_{t+1} = \frac{\la_t}{\mu_t} - 1 ,
\eta_t = (1 + \la_{t+1}) \al_t 
\ee
In the above, the quantity $\la_0$ is not specified and is chosen by the user.
They then define\footnote{To facilitate a comparison with the original paper,
we use the same symbol $\w_t$.
However, their quantity $\w_t$ is closer to our $\ubold_t$ defined
in Section \ref{ssec:41}.}
\be\label{eq:315}
\w_{t+1} = \w_t - \eta_t \h_{t+1} ,
\ee
\be\label{eq:316}
\bth_{t+1} = \frac{ \la_{t+1} }{ 1 + \la_{t+1} } \bth_t 
+ \frac{1}{1 + \la_{t+1} } \w_{t+1} .
\ee
Then $\bth_{t+1}$ satisfies \eqref{eq:313}.

The convergence of \eqref{eq:315}--\eqref{eq:316} is established under
\cite[Condition 1]{Sebbouh-et-al-CoLT21}, namely the sequence
$\{ \eta_t \}$ is decreasing, and moreover
\be\label{eq:311}
\sum_{t=0}^\infty \eta_t = \infty , \quad
\sum_{t=0}^\infty \eta_t^2 \s^2 < \infty , \quad
\sum_{t=1}^\infty \frac{\eta_t}{\sum_{\tau = 0}^{t-1} \eta_\tau } 
= \infty .
\ee
Thus in \cite{Sebbouh-et-al-CoLT21} the original step size sequence
$\{ \al_t \}$ and momentum sequence $\{ \mu_t \}$ are replaced by
the ``synthetic'' step size sequence $\{ \eta_t \}$, and the
convergence conditions are stated in terms of $\eta_t$.
It is shown that, in general, $J(\bth_t) \ap J^*$ where $J^*$ is
the minimum value of $J(\cdot)$, at a rate of $O(t^{-1/2})$.
In the ``over-parametrized'' case, the rate improves to $O(t^{-1})$.
Moreover, the iterations $\bth_t$ converge to a minimizer of $J(\cdot)$.

Now we give our interpretation of the results in \cite{Sebbouh-et-al-CoLT21}.
There are two restrictive features of these results.
First, the conditions \eqref{eq:311} are more stringent than the standard
Robbins-Monro conditions, namely
\be\label{eq:312}
\sum_{t=0}^\infty \eta_t^2 < \infty , \quad
\sum_{t=0}^\infty \eta_t = \infty,
\ee
Compared to \eqref{eq:312},
there are two extra assumptions in \eqref{eq:311}, namely: (i)
the synthetic step size $\eta_t$ is decreasing, and
 (ii) the summation of $\eta_t/\sum_{\tau = 0}^{t-1} \eta_\tau$
is divergent.
Since $S_t$ is an increasing sequence, the divergence of this summation
is a more restrictive assumption than the second Robbins-Monro condition
in \eqref{eq:312}

The second challenge in this approach is that, given the \textit{original}
step size and momentum sequences, there is no easy way to verify whether
\eqref{eq:311} is satisfied.
This is why, in \cite[Theorem 8]{Sebbouh-et-al-CoLT21}, the authors
\textit{begin with} the sequence $\{ \eta_t \}$, which appears to us to
be somewhat unnatural.
If $\mu_t = \mu$, a fixed constant, for all $t$,
then a possible solution to \eqref{eq:314} is
\bd
\la_t = \la_0 = \frac{\mu}{1 - \mu}, \fa t ,
\eta_t = \frac{1}{1-\mu} \al_t , \fa t .
\ed
Since $\eta_t$ is a constant multiple of $\al_t$, if
$\{ \al_t \}$ satisfies \eqref{eq:312}, then so does $\{ \eta_t \}$.
However, if $\mu_t$ varies as a function of $t$, this approach
will not work.
Specifically, it is shown in the Appendix that, if the momentum
coefficient $\mu_t$ is monotonically decreasing, then
$\la_t \ap \infty$ as $t \ap \infty$.
Consequently, a Robbins-Monro like assumpion such as
$\sum_{t=0}^\infty \al_t^2 < \infty$ \textit{need not imply} that
$\sum_{t=0}^\infty \eta_t^2 < \infty$.
In the other direction, if $\mu_t$ is monotonically increasing but
bounded away from $1$, them there exsts a finite $T$ such that
$1 + \la_{t+1} < 0$ for all $t \geq T$, thus causing the ``step size''
$\eta_t$ to become negative, which is absurd.

In contrast, the approach proposed here can handle the case where
not just the momentum parameter $\mu_t$ is time-varying, but 
\textit{all} parameters vary with $t$.
Moreover, the conditions for convergence reduce to the familiar Robbins-Monro
conditions if the stochastic gradient is unbiased and has finite variance
(even if the parameters vary with $t$).
In the more general case where the stochastic gradient is biased, and/or
the conditional variance of the stochastic gradient grows without bound
as a function of $t$, the conditions for convergence are gneralizations
of the Kiefer-Wolfowitz-Blum conditions.
This formulation allows us to handle the so-called zeroth-order methods,
wherein the stochastic gradient is computed using only noisy measurements
of the objective function.

Next we come to \cite{Liu-Yuan-arxiv22}.
The analysis in \cite{Sebbouh-et-al-CoLT21} is applicable only to
\textit{convex} objective functions.
In \cite{Liu-Yuan-arxiv22}, the authors prove results that are applicable
to arbitrary nonconvex functions that have a Lipschitz-continuous gradient.
They also relax the bound on the conditional variance of the stochastic
gradient to the so-called Expected Smoothness assumption of 
\cite{Khaled-Rich-arxiv20}, namely
\be\label{eq:317}
E_t( \nmeusq{\h_{t+1}}) \leq 2 A J(\bth_t) + B \nmeusq{\gJt} + C ,
\ee
for suitable constants $A, B, C$.
This is proposed in \cite{Khaled-Rich-arxiv20}
as ``the weakest assumption'' for analyzing
the convergence of SGD or SHB for nonconvex functions.
However, unlike in \cite{Sebbouh-et-al-CoLT21}, these authors assume that
the momentum term is a constant, that is, $\mu_t = \mu \fa t$.
Naturally, in this very general setting, one cannot expect to show
that the iterations converge almost surely to the set of minimizers.
Instead, the authors show that
\be\label{eq:318}
\lim_{\tai} \min_{0 \leq \tau \leq t} \nmeusq{\gJ(\bth_\tau)} = 0 .
\ee
Further, if the step size is chosen as $\al_t = c/(t^{0.5+\e})$
for some $\e \in (0,0.5)$, then
\be\label{eq:318a}
\min_{0 \leq \tau \leq t} \nmeusq{\gJ(\bth_\tau)} =
o \left( \frac{1}{t^{0.5-\e} } \right) .
\ee

Now we compare our results to those of \cite{Liu-Yuan-arxiv22}.
Throughout, we replace
the variance bound \eqref{eq:317} by weaker bound \eqref{eq:2110}.
We also permit the momentum parameter $\mu_t$ to vary with $t$, which
is not possible in the method of proof used in \cite{Liu-Yuan-arxiv22}.
When no convexity of any type is assumed, and the only assumption
is that $\gJ(\cdot)$ is Lipschitz-continuous, we are able to show that
\be\label{eq:318b}
\liminf_{\tai} \nmeu{\gJt} = 0.
\ee
Given any sequence of nonnegative numbers $\{ x_t \}$, it is easy to 
show that
\bd
\liminf_{\tai} x_t = 0 \imp \lim_{\tai} \min_{0 \leq \tau \leq t} x_\tau  = 0 ,
\ed
but the converse need not be true. 
(Suppose $x_T = 0$ for some $T$ but $x_t \geq \e > 0$ for all $t > T$.)
Hence our conclusion \eqref{eq:318b} is stronger than \eqref{eq:318a}.
Next, we permit a mild form of nonconvexity (namely the KL or PL properties).
In this more general setting, we nevertheless derive the almost sure
convergence of the iterations, when the Robbins-Monro or Kiefer-Wolfowitz-Blum
conditions are satisfied.

Next we come to \cite{Liu-et-al-NC23} and its predecessors
\cite{Zou-et-al-arxiv18,Yan-et-al-IJCAI18} that present a ``Stochastic
Unified Momentum (SUM)'' algorithm.
In the paper \cite{Liu-et-al-NC23}, the objective function
is of the same form as in (\cite[Eq.\ (1)]{Sebbouh-et-al-CoLT21}), namely
\bd
J(\bth) = E_{\w \sim P} F(\bth,\w) .
\ed
The  SUM algorithm consists of two coupled equations (in their notation):
\bd
m_t = \mu m_{t-1} - \eta_t g_t , \quad
x_{t+1} = x_t - \la \eta_t g_t + (1 - \la) m_t .
\ed
Other than the fact that the momentum coefficient $\mu$ is constant,
the only difference between the above, and \eqref{eq:122}--\eqref{eq:123},
is that the above has a ``convex combination'' of two terms, which is
absent in our formulation.
But this is a minor detail.
Hence it is not claimed that our unified algorithm itself is more general.
Rather, the generality is in the assumptions.
In \cite{Liu-et-al-NC23}, it is assumed that the stochastic gradient is
unbiased and has uniformly bounded variance, whereas we permit a more
general type of stochastic gradient, which satisfies
\eqref{eq:219}--\eqref{eq:2110}.
Our conclusions are also stronger.
Under the Robbins-Monro or Kiefer-Wolfowitz-Blum conditions,
when $J(\cdot)$ satisfies the (KL) property, 
we deduce that $\bth_t$ converges almost surely to the set of minimizers.
When $J(\cdot)$ satisfies the stronger (PL) property, we can bound
the rate of convergence.
Finally, if the only assumption is that $\gJ(\cdot)$ is Lipschitz-continuous,
we are able to show that
\bd
\liminf_{\tai} \nmeu{\gJt} = 0.
\ed
In contrast, in \cite{Liu-et-al-NC23}, the authors show only that
\bd
\lim_{\tai} \min_{1 \leq \tau \leq t} E[ \nmeusq{\gJ(\bth_\tau)} = 0 .
\ed
This is a weaker conclusion, as already discussed.

Finally we come to the closely related paper \cite{MV-TUKR-arxiv25},
for which the present author is a coauthor.
In a nutshell, the present paper extends \cite{MV-TUKR-arxiv25}
in several ways.
The salient points of difference between the present paper and
\cite{MV-TUKR-arxiv25} are the following:
\bit
\item In the present paper, we present a ``unified'' algorithm that
includes both the SHB and the SNAG as special cases.
In contrast, in \cite{MV-TUKR-arxiv25}, the focus is on the SHB algorithm
alone.
\item In the present paper, the momentum parameter is allowed to be
a function of time (the iteration counter), whereas it is assumed
to be a constant in \cite{MV-TUKR-arxiv25}.
Thus one contribution of this paper is to adapt the ``transformation
of variables'' approach in \cite{MV-TUKR-arxiv25} to the case where
the momemtum parameter varies with time.
When the momentum parameter is constant, the approach in 
\cite{MV-TUKR-arxiv25} leads to two completely decoupled updating equations.
A contribution of this paper is to show that it is good enough if the
transformed equations are ``asymptotically'' decoupled.
\item In the present paper, the approach proposed
in \cite{Sebbouh-et-al-CoLT21} is analyzed in detail;
it is shown that the approach in \cite{Sebbouh-et-al-CoLT21}
leads to considerably stronger
sufficient conditions than in the present paper.
The paper \cite{MV-TUKR-arxiv25} does not attempt to analyze
\cite{Sebbouh-et-al-CoLT21}.
\eit

\section{Main Results}\label{sec:Main}

In this section we state the main results of the paper.
The proofs are deferred to Section \ref{sec:Proofs}.

\subsection{Various Assumptions}\label{ssec:21}

In this subsection, we state the various assumptions made in order to
prove the main results.
They are of three kinds: On the objective function, on the stochastic
gradient, and on the various constants.

\subsubsection{Assumptions on the Objective Function}\label{sssec:211}

We begin with two ``standing'' assumptions on the objective function
$J(\cdot)$ and its gradient $\gJ(\cdot)$.
Note that these assumptions are made in almost every paper in this area.
\ben
\item[(S1)] $J(\cdot)$ is $\C^1$, and $\gJ(\cdot)$ is globally 
Lipschitz-continuous with constant $L$.
\item[(S2)] $J(\cdot)$ is bounded below, and the infimum is attained.
Thus
\bd
J^* := \inf_{\bth \in \R^d} J(\bth) > - \infty .
\ed
Further, the set $S_J$ defined by
\be\label{eq:214}
S_J := \{ \bth : J(\bth) = 0 \}
\ee
is nonempty.
\een

Next we present a useful consequence of Assumptions (S1)
and (S2).
This is the same as \cite[Lemma 1]{MV-RLK-SGD-JOTA24} and the proof
can be found therein.
Note that the validity of the lemma for \textit{convex} functions
is established in \cite[Theorem 2.1.5]{Nesterov04}.
See specifically Eq.\ (2.1.6) therein, which is shown to be equivalent
to convexity.
Thus the advantage of Lemma \ref{lemma:21} is that the
bound \eqref{eq:211} holds even when $J(\cdot)$ is not convex.

\begin{lemma}\label{lemma:21}
Suppose (S1) and (S2) hold.\footnote{Note that the result holds
whenever $J(\cdot)$ is bounded below, even if the infimum is not attained.
The proof can be found in \cite[Lemma 1]{MV-RLK-SGD-JOTA24}.}
Then
\be\label{eq:211}
\nmeusq{ \nabla J(\bth) } \leq 2L J(\bth) .
\ee
\end{lemma}

We also require a very useful bound, namely \cite[Eq.\ (2.4)]{Ber-Tsi-SIAM00}.
This lemma is used repeatedly in the sequel.
Note that if $J(\cdot)$ belongs to $\C^2$, then the
lemma is a ready consequence of Taylor's theorem.
The advantage of Lemma \ref{lemma:42} is that $J(\cdot)$ is assumed
only to belong to $\C^1$, but with $\gJ(\cdot)$ being globally 
Lipschitz-continuous, which implies that $\gJ(\cdot)$ is absolutely
continuous and thus differentiable \textit{almost} everywhere.
This is a somewhat slight, but useful, relaxation.

\begin{lemma}\label{lemma:42}
Suppose $J: \R^d$ is $\C^1$, and that $\gJ(\cdot)$ is $L$-Lipschitz
continuous.
Then
\be\label{eq:432a}
J(\bth+\bphi) \leq J(\bth) + \IP{\gJ(\bth)}{\bphi} + \frac{L}{2} 
\nmeusq{\bphi} , \fa \bth , \bphi \in \R^d .
\ee
\end{lemma}

Next we introduce three properties (or assumptions), known respectively as
the Polyak-{\L}ojasiewicz (PL) property, 
the (KL') property, which is a slight relaxation
of the well-known Kurdyka-{\L}ojasiewicz property,
and the (NSC) property.
As a prelude, we define the class of functions of Class $\B$.
Again, all of this material can be found in 
\cite[Section 4]{MV-RLK-SGD-JOTA24}.

\begin{definition}\label{def:Class-B}
A function $\eta : \R_+ \ap \R_+$ is
said to \textbf{belong to Class $\B$} if $\eta(0) = 0$, and in addition,
for arbitrary real numbers $0 < \e \leq M$,
it is true that
\bd
\inf_{\e \leq r \leq M} \eta(r) > 0 .
\ed
\end{definition}

Using this concept, we introduce a few other properties on $J(\cdot)$.

\ben
\item[(PL)] There exists a constant $K$ such that
\be\label{eq:212}
\nmeusq{\gJ(\bth)} \geq K J(\bth) , \fa \bth \in \R^d .
\ee
\item[(KL')] There exists a function $\psi(\cdot)$ of Class $\B$
such that
\be\label{eq:213}
\nmeu{\gJ(\bth)} \geq \psi(J(\bth)) , \fa \bth \in \R^d .
\ee
\item[(NSC)]
This property consists of the following assumptions, taken together.
\ben
\item
The function $J(\cdot)$ has compact level sets.
For every constant $c \in (0,\infty)$, the level set
\bd
L_J(c) := \{ \bth \in \R^d : J(\bth) \leq c \}
\ed
is compact.
\item
There exists a number $r > 0$ and a continuous function
$\eta : [0,r] \ap \R_+$ such that $\eta(0) = 0$, and
\be\label{eq:215}
\rho(\bth) \leq \eta(J(\bth)) , \fa \bth \in L_J(r) .
\ee
\een
\een

PL stands for the Polyak-{\L}ojasiewicz condition.
In \cite{Polyak-UCMMP63}, Polyak introduced \eqref{eq:212},
and showed that it is sufficient to ensure that iterations converge at
a ``linear'' (or geometric) rate to a global minimum, whether or not
$J(\cdot)$ is convex.
The property (KL') is weaker than (PL).
A good review of these properties can be found in \cite{Karimi-et-al16}.
(NSC) stands for ``nearly strongly convex.''
Every strongly convex function satisfies this property, but so do others.
It is obvious that, if (NSC) is satisfied, then $J(\bth_t) \ap 0$ as $\tai$
implies that $\rho(\bth_t) \ap 0$ as $\tai$.

\subsubsection{Assumptions on the Stochastic Gradient}\label{sssec:212}

Let $\F_t$ denote the $\s$-algebra generated by $\bth_0 , \h_1^t$,
where $\h_1^t$ denotes $(\h_1 , \cdots \h_t)$;
note that there is no $\h_0$.
For an $\R^d$-valued
 random variable $X$, let $E_t(X)$ denote the \textbf{conditional
expectation} $E(X | \F_t)$, and let $CV_t(X)$ denote its 
\textbf{conditional variance} defined by\footnote{See 
\cite{Williams91,Durrett19} for relevant background on stochastic processes.}
\be\label{eq:216}
CV_t(X) = E_t( \nmeusq{ X - E_t(X)} ) = E_t(\nmeusq{X}) - \nmeusq{E_t(X)} .
\ee

With these notational conventions in place we state the assumptions
on $\h_{t+1}$.
We begin by defining
\be\label{eq:217}
\z_t = E_t(\h_{t+1}) , \quad
\x_t = \z_t - \gJw ,\quad
\bzt = \h_{t+1} - \z_t .
\ee
Thus $\x_t$ denotes the ``bias'' of the stochastic gradient.
If $\h_{t+1}$ is an unbiased estimate of $\gJw$, then $\x_t = \bz$.
Most papers in the literature assume that $\x_t = \bz$, but
our objective here is specifically to \textit{permit} biased estimates.
This is necessary to analyze the situation where the stochastic gradient
is obtained using function valuations alone.
The last equation in \eqref{eq:217} implies that $E_t(\bzt) = \bz$.
Therefore 
\be\label{eq:218}
E_t( \nmeusq{\h_{t+1}} ) = \nmeusq{\z_t} + E_t ( \nmeusq{\bzt} ) .
\ee

With these definitions, the assumption on the stochastic gradient is that
there exist sequences of constants $\{ B_t \}$ and $\{ M_t \}$ such that
\be\label{eq:219}
\nmeu{\x_t} \leq B_t [ 1 + \nmeu{\gJw} ] , \fa \bth_t \in \R^d , \fa t ,
\ee
\be\label{eq:2110}
E_t ( \nmeusq{\bzt}) \leq M_t^2 [ 1 + J(\w_t) ] ,
\fa \bth_t \in \R^d , \fa t .
\ee
Equation \eqref{eq:219} states that the stochastic gradient $\h_{t+1}$
can be biased, but the extent of the bias has to be bounded by a constant
plus the norm of the gradient.
As we will see in subsequent sections, while $B_t$ is permitted to be nonzero,
eventually it has to approach zero; in other words, the stochastic gradient
has to be ``asymptotically unbiased.''
In contrast, \eqref{eq:2110} states that the conditional variance of
the stochastic gradient can grow as a function of the iteration counter $t$.
This feature is essential to permit the analysis of so-called zeroth-order
methods, where only a small number (often just two) of
function evaluations are used to construct $\h_{t+1}$.

\subsubsection{Assumptions on the Constants}\label{sssec:213}

Aside from the step length $\al_t$, there are four constants
in the algorithm \eqref{eq:121}--\eqref{eq:122}.
The assumptions on these constants are as follows:
There exist constants $\ab,\bul,\bb,\mub,\ebar$ such that, for all $t$, we have
\be\label{eq:2111}
0 \leq a_t \leq \ab , 0 < \bul \leq b_t \leq \bb , 0 \leq \mu_t \leq \mub < 1 ,
| \e_t | \leq \ebar < \infty .
\ee
Now we discuss a few implications of the above bounds.
First, $a_t$ is always nonnegative and bounded above.
Second, $b_t$ is bounded both below and above by positive
constants.
Third, the momentum coefficient $\mu_t$ can equal zero, but is bounded
away from $1$.
Finally, $\e_t$ can be either positive or negative, but is bounded in
magnitude.
Observe that when SHB is formulated as a special case of
\eqref{eq:121}--\eqref{eq:122}, the assumptions in \eqref{eq:2111} hold.
As for SNAG, in the \textit{traditional} formulation, the momentum
parameter $\mu_t \uparrow 1$ as $\tai$.
Hence the assumptions in \eqref{eq:2111} do not hold.
What is analyzed here is a nonstandard version of SNAG in which
\eqref{eq:2111} hold.
The version of SNAG analyzed in \cite{Liu-Yuan-arxiv22}
is even more restrictive in that $\mu_t$ is a fixed constant less than one.

Two ready consequences of these assumptions are that, if we define
\be\label{eq:2111a}
k_t := \frac{a_t}{(1-\mu_t)} , \kb := \frac{\ab}{1 - \mub} ,
\ee
then
\be\label{eq:2112}
k_t \in [0,\kb] , b_t + k_{t+1} \in [\bul , \bb+ \ab/(1-\mub)] .
\ee

A key assumption is this:
Define $\de_t := k_{t+1} - k_t$.
Then
\be\label{eq:2113}
\de_t \ap 0 \mbox{ as } \tai .
\ee
Note that there are no restrictions on the \textit{sign} of $\de_t$.
This assumption is readily satisfied if both $\{ a_t \}$ and $\{ \mu_t \}$
converge to some limits.
The assumption allows us to transform the variables
in \eqref{eq:121}--\eqref{eq:122} in such a way that the resulting
transformed equations are ``asymptotically decoupled.''
More details can be found in Section \ref{ssec:41}.

In our analysis, it is quite permissible to allow \textit{all five}
constants $a_t$, $ b_t$, $ \e_t $, $ \mu_t $, $ \al_t$ to be random variables.
In this case, the bounds in \eqref{eq:2111} and \eqref{eq:2112}
hold almost surely.
If we define $\F_t$ to be the $\s$-algebra generated by $\bth_0$ and
$\h_1^t$, then all of these constants need to belong to $\M(\F_t)$,
the set of random variables that are measurable with respect to $\F_t$.
In particular, in \eqref{eq:1212}, we see that $\e_t = \mu_{t+1} \mu_t$.
Thus, in order to incorporate the approach of \cite{Bengio-et-al-ICASSP13}
in the present framework,
we must assume that $\mu_{t+1} \in \M(\F_t)$, i.e., that $\{ \mu_t \}$
is a \textbf{predictable} process.

\subsection{Main Theorems with Full Coordinate Updating}\label{ssec:22}

In this subsection we state the two main theorems regarding the convergence
of the general algorithm \eqref{eq:121} and \eqref{eq:122}),
and several corollaries thereof.
In summary, when the objective function $J(\cdot)$ satisfies the (KL')
property, and the analogs of the Kiefer-Wolfowitz-Blum conditions
are satisfied (see \eqref{eq:221} and \eqref{eq:222} below),
then the algorithm converges almost surely.
If the hypothesis on $J(\cdot)$ is strengthened to (PL) from (KL'),
then we can also derive \textit{bounds} on the rate of convergence.

\begin{theorem}\label{thm:21}
Suppose that the various constants satisfy the assumptions in Section
\ref{sssec:213}, 
while the objective function $J(\cdot)$ satisfies Standing Assumptions
(S1) and (S2) in Section \ref{sssec:211}.
Further, suppose the stochastic gradient $\h_{t+1}$ satisfies the
assumptions \eqref{eq:219}--\eqref{eq:2110} in Section \ref{sssec:212}.
With these assumptions, we can state the following:
\ben
\item Suppose
\be\label{eq:221}
\sum_{t=0}^\infty \al_t^2 < \infty , \quad 
\sum_{t=0}^\infty \al_t B_t  < \infty , \quad 
\sum_{t=0}^\infty \al_t^2 M_t^2 < \infty .
\ee
Then $\{ \gJ(\bth_t) \}$ and $\{ J(\bth_t) \}$ are bounded, and in addition,
$J(\bth_t)$ converges almost surely to some random variable as $\tai$.
\item If in addition 
\be\label{eq:222}
\sum_{t=0}^\infty \al_t = \infty ,
\ee 
then
\be\label{eq:222a}
\liminf_{\tai} \nmeu{\gJt} = 0 .
\ee
\item If, in addition to \eqref{eq:221} and \eqref{eq:222}, the function
$J(\cdot)$ satisfies (KL'), 
then $J(\bth_t) \ap 0$ and $\gJ(\bth_t) \ap \bz$ as $\tai$, where both
convergences are in the almost sure sense.
\item Suppose that in addition to (KL'), $J(\cdot)$ also satisfies (NSC),
and that \eqref{eq:221} and \eqref{eq:222} both hold.
Then $\rho(\bth_t) \ap 0$ almost surely as $\tai$.
\een
\end{theorem}

Now we state some useful corollaries of the above theorem.

\begin{corollary}\label{coro:211}
Suppose that the various constants satisfy the assumptions in Section
\ref{sssec:213},
while the objective function $J(\cdot)$ satisfies Standing Assumptions
(S1) and (S2) in Section \ref{sssec:211}.
Further, suppose the stochastic gradient $\h_{t+1}$ satisfies the
assumptions \eqref{eq:219}--\eqref{eq:2110} in Section \ref{sssec:212},
with $B_t = 0$ for all $t$, and $M_t^2 \leq M^2$ for all $t$
for some fixed constant $M$.
With these assumptions, we can state the following:
\ben
\item Suppose
\be\label{eq:223}
\sum_{t=0}^\infty \al_t^2 < \infty .
\ee
Then $\{ \gJ(\bth_t) \}$ and $\{ J(\bth_t) \}$ are bounded, and in addition,
$J(\bth_t)$ converges almost surely to some random variable as $\tai$.
\item If in addition \eqref{eq:222} holds, then 
\bd
\liminf_{\tai} \nmeu{\gJt} = 0 .
\ed
\item If in addition $J(\cdot)$ satisfies (KL'), 
then $J(\bth_t) \ap 0$ and $\gJ(\bth_t) \ap \bz$ as $\tai$, where both 
convergences are in the almost sure sense.
\item Suppose that in addition to (KL'), $J(\cdot)$ also satisfies (NSC),
and that \eqref{eq:223} and \eqref{eq:222} both hold.
Then $\rho(\bth_t) \ap 0$ almost surely as $\tai$.
\een
\end{corollary}

Note that \eqref{eq:223} and \eqref{eq:221} are the familiar
Robbins-Monro conditions introduced in \cite{Robbins-Monro51}.
Thus, when the stochastic gradient is unbiased and has bounded variance,
the conditions for the convergence of the general algorithm 
\eqref{eq:121}--\eqref{eq:122} are the familiar ones for SGD, as
shown in \cite{MV-RLK-SGD-JOTA24}.

\begin{corollary}\label{coro:212}
Under the assumptions of Theorem \ref{thm:21}, suppose further that there
exists a sequences of constants $c_t$ (known as the ``increment'') such that
$B_t = O(c_t)$, and $M_t^2 = O(1/c_t^2)$.
With these assumptions, we can state the following:
\ben
\item Suppose
\be\label{eq:224}
\sum_{t=0}^\infty \al_t^2 < \infty , \quad
\sum_{t=0}^\infty \al_t c_t < \infty , \quad 
\sum_{t=0}^\infty (\al_t^2)/(c_t^2) < \infty .
\ee
Then $\{ \gJ(\bth_t) \}$ and $\{ J(\bth_t) \}$ are bounded, and in addition,
$J(\bth_t)$ converges almost surely to some random variable as $\tai$.
\item If in addition $J(\cdot)$ satisfies (KL'), and
\eqref{eq:222} holds,
then $J(\bth_t) \ap 0$ and $\gJ(\bth_t) \ap \bz$ as $\tai$,
where both convergences are in the almost sure sense.
\item Suppose that in addition to (KL'), $J(\cdot)$ also satisfies (NSC),
and that \eqref{eq:223} and \eqref{eq:222} both hold.
Then $\rho(\bth_t) \ap 0$ almost surely as $\tai$.
\een
\end{corollary}

Thus the point of these two corollaries is to show that the conditions
for convergence in Theorem \ref{thm:21} are natural generalizations
of standard results for stochastic gradient descent, as proved in
\cite{MV-RLK-SGD-JOTA24}, even in the presence of time-varying momentum terms.
In contrast, as shown in the Appendix, the previously
known sufficient conditions for convergence 
given in \cite{Sebbouh-et-al-CoLT21} are more restrictive.

Corollary \ref{coro:212} is relevant when the stochastic gradient is obtained
using only function evaluations, and no gradient computations.
This approach was frst introduced for the case $d=1$ in \cite{Kief-Wolf-AOMS52},
where it is established that we can take $B_t = O(c_t)$ and $M_t = O(1/c_t)$,
where $c_t$ is the increment.
The approach was extended to the case of arbitrary $d$ in \cite{Blum54}.
In \cite{Kief-Wolf-AOMS52}, the conditions \eqref{eq:224}--\eqref{eq:222}
are derived as sufficient conditions for convergence, and a different
proof is given in \cite{Blum54}.
Thus \eqref{eq:224}--\eqref{eq:222} are usually referred to as the
Kiefer-Wolfowitz-Blum conditions.

Blum's approach required $d+1$ function evaluations, which can be troublesome
if $d$ is very large.
In a series of papers, an alternate approach that requires \textit{only two}
function evaluations was proposed, under the name of Simultaneous Perturbation
Stochastic Perturbation (SPSA).
The latest in this series of papers is \cite{Sadegh-Spall-TAC98}.
In this paper, the perturbation is by a $d$-dimensional vector of
Rademacher variables, that is, pairwise independent random variables
that assume the values $\pm$ with equal probability.
Let $\bD_{t+1} \in \bp^d$ denote the vector of Rademacher variables
at time $t+1$.
Then the search direction $\h_{t+1}$ in \eqref{eq:121} is defined
componentwise, via
\be\label{eq:224a}
h_{t+1,i} = \frac{[ J(\bth_t + c_t \bD_{t+1}) + \xi_{t+1,i}^+ ]
- [ J(\bth_t - c_t \bD_{t+1}) - \xi_{t+1,i}^- ] } {2 c_t \D_{t+1, i}} ,
\ee
where  $\xi_{t+1,1}^+ , \cdots , \xi_{t+1,d}^+$,
$\xi_{t+1,1}^- , \cdots , \xi_{t+1,d}^-$ represent the measurement errors.
Here, $c_t$ denotes the multiplier of the random vector, which
is called the ``increment.''
In this case, it can be shown that $B_t = O(c_t)$, and $M_t^2 = O(1/c_t^2)$,
and Corollary \ref{coro:212} applies.
A similar idea is used in \cite{Nesterov-FCM17},
except that the bipolar vector $\bD_{t+1}$ is replaced by a random
Gaussian vector.

The objective of the next theorem is to show that if the hypothesis (KL')
is strengthened to (PL), then it is possible to obtain bounds on the
\textit{rate} of convergenc e.

\begin{theorem}\label{thm:22}
Let various symbols be as in Theorem \ref{thm:21}.
Suppose $J(\cdot)$ satisfies the standing assumptions (S1) and (S2)
and also property (PL),
and that \eqref{eq:224} and \eqref{eq:222} hold.
Further, suppose there exist constants $\g > 0$ and $\de \geq 0$ such
that
\bd
B_t = O(t^{-\g}), \quad
M_t = O(t^\de) , \fa t \geq 1 ,
\ed
where we take $\g = 1$ if $B_t = 0$ for all sufficiently large $t$,
and $\de = 0$ if $M_t$ is bounded.
Choose the step-size sequence $\{ \al_t \}$ as
$O(t^{-(1-\phi)})$ and $\OM(t^{-(1-C)})$
where $\phi$ and $C$ are chosen to satisfy
\be\label{eq:225}
0 < \phi < \min \{ 0.5 - \de , \g \} , \quad
C \in (0,\phi] .
\ee
Define
\be\label{eq:226}
\nu := \min \{ 1 - 2( \phi + \de) , \g - \phi \} .
\ee
Then $\nmeusq{\gJt} = o(t^{-\la})$ and $J(\bth_t) = o(t^{-\la})$
for every $\la \in (0,\nu)$.
In particular, by choosing $\phi$ very small, it follows that
$\nmeusq{\gJt} = o(t^{-\la})$ and $J(\bth_t) = o(t^{-\la})$ whenever
\be\label{eq:227}
\la < \min \{ 1 - 2 \de , \g \} .
\ee
\end{theorem}

Until now we have studied \textit{full coordinate updating}, whereby
the update equations \eqref{eq:121}--\eqref{eq:122} are applied to
\textit{every} coordinate of $\w_t$ and $\vbold_t$.
Next, we invoke a ``meta-theorem'' from \cite[Section 4]{MV-TUKR-arxiv25}
to state a theorem on \textit{block updating}, whereby, at each step $t$,
\textit{some but not necessarily all} components of $\w_t$
and $\vbold_t$ are updated.
Let $S_t \seq [d]$ denote the components of $\w_t$ and of $\vbold_t$
that are updated at each $t$.
Then both the cardinality and the elements of $S_t$
can be random, and can vary with $t$.

In \cite[Section 4]{MV-TUKR-arxiv25}, three different methods of block
updating are studied, which are now reprised for the convenience of the reader.
Note that Option 1 is the full-coordinate update.

\textbf{Option 2: Single Coordinate Update:}
At time $t$, choose an index $\kappa_t \in [d]$ at random with a
uniform probability, and independently of previous choices.
Let $\eb_{\kappa_t}$ denote the elementary unit vector with a $1$
as the $\kappa_t$-th component and zeros elsewhere.
Then define
\be\label{eq:232}
\h^{(2)}_{t+1} = d \eb_{\kappa_t } \circ \h_{t+1} ,
\ee
where $\circ$ denotes the Hadamard, or component-wise, product of two
vectors of equal dimension.

\textbf{Option 3: Multiple Coordinate Update:}
This option is just coordinate update along multiple
coordinates chosen independently at random.
At time $t$, choose $N$ different indices $\kappa_t^n$ from $[d]$
\textit{with replacement}, with each choice being independent of the rest,
and also of past choices.
Moreover, each $\kappa_t^n$ is chosen from $[d]$ with uniform probability.
Then define
\be\label{eq:233}
\h^{(3)}_{t+1} := \frac{d}{N} \sum_{n=1}^N \eb_{\kappa_t^n}
\circ \h_{t+1} .
\ee

\textbf{Option 4: Bernoulli Update:}
At time $t$, let $\{ B_{t,i}, i \in [d] \}$  be independent Bernoulli
processes with success rate $\rho_t$.
Thus
\be\label{eq:234}
\Pr \{ B_{t,i} = 1 \} = \rho_t , \fa i \in [d] .
\ee
It is permissible for the success probability $\rho_t$ to vary with time.
However, at any one time, all components must have the same success
probability.
Then define
\be\label{eq:235}
\vbold_t := \sum_{i=1}^d \eb_i I_{ \{ B_{t,i} = 1 \} }  \in \bi^d .
\ee
Thus $\vbold_t$ is a random vector, and
$v_{t,i}$ equals $1$ if $ B_{t,i} = 1$, and equals $0$ otherwise.
Now define
\be\label{eq:236}
\h^{(4)}_{t+1} = \frac{1}{\rho_t}  \vbold_t \circ \h_{t+1} .
\ee

It is necessary to clear up one important aspect of block updating.
It is \textit{not the case} that the full stochastic gradient $\h_{t+1}$
is computed first , and the set $S_t$ of components to be updated is
determined afterwards.
In that approach, there would not be any savings in computational effort.
Rather, the set $S_t$ is determined \textit{first}, and afterwards
only the components $h_{i,t+1}, i \in S_t$ are computed.
If a method such as back-propagation is used, then computing
\textit{only some} components of a gradient is not noticeably cheaper
than computing the full gradient.
However, if the SPSA approach is used, then \eqref{eq:224a} can be applied,
by choosing the Rademacher variable $\D_{i,t} = 0$ for $i \not\in S_t$.
If $|S_t| \ll d$, this approach results in considerable savings in
computation.

With all these preliminaries, we now state the theorem on block updating.
\begin{theorem}\label{thm:23}
Suppose the stochastic gradient $\h_{t+1}$ satisfies the bounds
\eqref{eq:219} and \eqref{eq:2110}.
Suppose that in \eqref{eq:122}--\eqref{eq:123}, the quantity $\h_{t+1}$ is replaced by
$\h_{t+1}^{(k)}$ for $k = 2, 3, 4$.
Further, suppose that when Option $4$ is used, then
\bd
\inf_t \rho_t =: \bar{\rho} > 0 .
\ed
Then the conclusions of Theorem \ref{thm:21} or Theorem \ref{thm:22}
continue to hold under the same assumptions on $J(\cdot)$ and the
same assumptions on $\{ \al_t \}$.
\end{theorem}

Because the theorem is a straight-forward application of
\cite[Theorem 8]{MV-TUKR-arxiv25} once Theorem \ref{thm:22} is proved,
the proof of this theorem is omitted.

\section{Proofs of the Main Results}\label{sec:Proofs}

\subsection{Transformation of Variables}\label{ssec:41}

The convergence analysis of \eqref{eq:121}--\eqref{eq:122} is based on
carrying out a linear transformation of the variables such that
the resulting equations are ``nearly'' decoupled, and are \textit{exactly}
decoupled if all terms $a_t, b_t, \e_t , \mu_t$ are constant.
In contrast, in \cite{Sebbouh-et-al-CoLT21}, the authors propose a 
linear transformation
that achieves exact decoupling even when $\mu_t$ varies with $t$.
As shown in the Appendix, this approach is \textit{untenable} when
$\mu_t$ is monotonic, either decreasing or increasing.
In contrast, our approach does not suffer from such limitations.
Moreover, as shown in the results stated in Section \ref{sec:Main},
our conditions
for the convergence of the algorithm in \eqref{eq:121}--\eqref{eq:122} 
are natural generalizations of the familiar Robbins-Monro
\cite{Robbins-Monro51} or the Kiefer-Wolfowitz-Blum
\cite{Kief-Wolf-AOMS52,Blum54} conditions, unlike in
\cite{Sebbouh-et-al-CoLT21}.

Let us rewrite \eqref{eq:121}--\eqref{eq:122} as
\be\label{eq:411}
\left[ \ba{c} \w_{t+1} \\ \vbold_{t+1} \ea \right] =
\left[ \ba{cc} I & a_t I \\ 0 & \mu_t I \ea \right]
\left[ \ba{c} \w_t \\ \vbold_t \ea \right] - 
\left[ \ba{c} b_t I \\ I \ea \right] \al_t \h_{t+1} ,
\ee
where each $I$ denotes $I_{d \times d}$.
Define
\be\label{eq:412}
A_t = \left[ \ba{cc} I & a_t I \\ 0 & \mu_t I \ea \right] ,
\LA_t = \left[ \ba{cc} I & 0 \\ 0 & \mu_t I \ea \right] .
\ee
Then $A_t$ is the coefficient matrix in \eqref{eq:411} and $\LA_t$
is the matrix of the eigenvalues of $A_t$.
In order to diagonalize $A_t$ into $\LA_t$, we compute the matrix of
eigenvectors of $A_t$, as follows:
\be\label{eq:413}
Z_t = \left[ \ba{cc} I & -\frac{a_t}{1 - \mu_t} I \\ 0 & \mu_t I \ea \right] 
= \left[ \ba{cc} I & -k_t I \\ 0 & \mu_t I \ea \right] ,
Z_t^{-1} = \left[ \ba{cc} I & k_t I \\ 0 & \mu_t I \ea \right] ,
\ee
where as already defined in \eqref{eq:2111a}, we have that
\be\label{eq:413a}
k_t = \frac{a_t}{1 - \mu_t} .
\ee
Then $Z_t^{-1} A_t Z_t = \LA_t$.
Next, define
\be\label{eq:414}
\left[ \ba{c} \ubold_t \\ \vbold_t \ea \right] :=
Z_t^{-1} \left[ \ba{c} \w_t \\ \vbold_t \ea \right] 
= \left[ \ba{c} \w_t + k_t \vbold_t \\ \vbold_t \ea \right] .
\ee
Here we take advantage of the fact that the bottom block of $Z_t^{-1}$ is
$[0 \; \; I]$.
Hence, in effect, $\w_t$ is replaced by $\ubold_t$, but $\vbold_t$
is left unaltered.
Hence the update equation for $\vbold_t$ also remains as \eqref{eq:122}.

Next we compute the update equation for $\ubold_t$.
\be\label{eq:415}
\ubold_{t+1} = \w_{t+1} + k_{t+1} \vbold_{t+1} 
= \w_{t+1} + k_t \vbold_{t+1} + \de_t \vbold_{t+1} ,
\ee
where
\be\label{eq:416}
\de_t = k_{t+1} - k_t = \frac{a_{t+1}}{1 - \mu_{t+1}} - \frac{a_t}{1 - \mu_t} .
\ee
Now observe that
\bd
\w_{t+1} + k_t \vbold_{t+1} = \w_t + a_t \vbold_t - b_t \al_t \h_{t+1}
+ k_t \mu_t \vbold_t - k_t \al_t \h_{t+1} .
\ed
However
\bd
k_t \mu_t + a_t = a_t \left( \frac{\mu_t}{1-\mu_t} + 1 \right)
= \frac{a_t}{1 - \mu_t} = k_t .
\ed
Hence we can write
\be\label{eq:417}
\w_{t+1} + k_t \vbold_{t+1} = \w_t + k_t \vbold_t - \al_t ( b_t + k_t )
\h_{t+1} = \ubold_t - \al_t ( b_t + k_t ) \h_{t+1} .
\ee
The last term in \eqref{eq:415} becomes
\be\label{eq:418}
\de_t \vbold_{t+1} = \de_t \mu_t \vbold_t - \de_t \al_t \h_{t+1} .
\ee
Substituting from \eqref{eq:417} and \eqref{eq:418} into \eqref{eq:415}
gives the final form of the update equation for $\ubold_t$.
\be\label{eq:419}
\begin{split}
\ubold_{t+1} &= \ubold_t + \de_t \mu_t \vbold_t - ( b_t + k_t + \de_t )
\al_t \h_{t+1} \\
&= \ubold_t + \de_t \mu_t \vbold_t - (b_t + k_{t+1}) \al_t \h_{t+1} ,
\end{split}
\ee
while the updating equation for $\vbold_t$ remains as before, namely
\be\label{eq:4110}
\vbold_{t+1} = \mu_t \vbold_t - \al_t \h_{t+1} .
\ee
These are the two equations whose behavior is analyzed in the remainder
of the paper.
Based on the analysis, we make inferences about the behavior $\w_t$,
and eventually, $\bth_t$.
Note that these two equations are \textit{not} decoupled in general, due to the
presence of the term $\de_t \mu_t \vbold_t$ in \eqref{eq:419}.
However, in the special case where both $a_t$ and $\mu_t$ are constant,
then $\de_t = 0$ for all $t$, and the equations are indeed decoupled.
This is the approach used in \cite{MV-TUKR-arxiv25} to study the SHB
algorithm when $\mu_t$ is constant.
More generally, if both $a_t$ and $\mu_t$ converge to some some constants
as $\tai$, then $\de_t \ap 0$ as $\tai$, and the equations become
``asymptotically decoupled.''
We can draw some useful conclusions when $\de_t \ap 0$ as $\tai$.

\subsection{The Robbins-Siegmund Theorem and Some Extensions}\label{ssec:42}

The proof of Theorem \ref{thm:21} makes use of
a fundamental result in the convergence of stochastic processes,
known as the ``almost supermartingale'' theorem due to
Robbins and Siegmund \cite[Theorem 1]{Robb-Sieg71}.
It is also found in \cite{BMP90} and in \cite{Fran-Gram21}.
The theorem states the following:

\begin{lemma}\label{lemma:41}
Suppose $\{ V_t \} , \{ f_t \} , \{ g_t \} , \{ h_t \}$ are
stochastic processes taking values in $[0,\infty)$, adapted to some
filtration $\{ \F_t \}$, satisfying
\be\label{eq:421}
E_t( V_{t+1} ) \leq (1 + f_t) V_t + g_t - h_t \as, \fa t ,
\ee
where, as before, $E_t(V_{t+1})$ is a shorthand for $E(V_{t+1} | \F_t )$.
Then, on the set
\bd
\OM_0 := \{ \om \in \OM : \sum_{t=0}^\infty f_t(\om) < \infty \}
\cap \{ \om : \sum_{t=0}^\infty g_t(\om) < \infty \} ,
\ed
we have that $\lim_{\tai} V_t$ exists, and in addition,
$\sum_{t=0}^\infty h_t(\om) < \infty$.
In particular, if $P(\OM_0) = 1$, then $\{ V_t \}$ is bounded
almost surely, and $\sum_{t=0}^\infty h_t(\om) < \infty$ almost surely.
\end{lemma}

The following theorem is a straight-forward, but useful extension
of Lemma \ref{lemma:41}.
It is Theorem 5.1 of \cite{MV-RLK-SGD-arxiv23,MV-RLK-SGD-JOTA24}.

\begin{theorem}\label{thm:41}
Suppose $\{ V_t \} , \{ f_t \} , \{ g_t \} , \{ h_t \}, \{ \al_t \}$ are
$[0,\infty)$-valued stochastic processes
defined on some probability space $(\OM,\SI,P)$, and
adapted to some filtration $\{ \F_t \}$.
Suppose further that
\be\label{eq:422}
E_t(V_{t+1} ) \leq (1 + f_t) V_t + g_t - \al_t h_t \as, \fa t .
\ee
Define
\be\label{eq:423}
\OM_0 := \{ \om \in \OM : \sum_{t=0}^\infty f_t(\om) < \infty \mbox{ and }
\sum_{t=0}^\infty g_t(\om) < \infty \} ,
\ee
\be\label{eq:424}
\OM_1 := \{ \sum_{t=0}^\infty \al_t(\om) = \infty \} .
\ee
Then
\ben
\item Suppose that $P(\OM_0) = 1$.
Then the sequence $\{ V_t \}$ is bounded almost surely, and
there exists a random variable $W$ defined on $(\OM,\SI,P)$ such that
$V_t(\om) \ap W(\om)$ almost surely.
\item
Suppose that, in addition to $P(\OM_0) = 1$, it is also true that
$P(\OM_1) = 1$.
Then
\be\label{eq:425}
\liminf_{\tai} h_t(\om) = 0  \fa \om \in \OM_0 \cap \OM_1 .
\ee
\item
Further, suppose there exists a function $\eta(\cdot)$ of Class $\B$
such that $h_t(\om) \geq \eta(V_t(\om))$ for all $\om \in \OM_0$.
Then $V_t(\om) \ap 0$ as $\tai$ for all $\om \in \OM_0 \cap \OM_1$.
\een
\end{theorem}

Theorem \ref{thm:21} allows us to infer convergence, but does not
provide any information about the \textit{rate} of convergence.
Now we define the concept of a rate of convergence of stochastic
processes, following a similar definition in \cite{Liu-Yuan-arxiv22}.

\begin{definition}\label{def:order}
Suppose $\{ Y_t \}$ is a stochastic process, and $\{ f_t \}$
is a sequence of positive numbers.
We say that
\ben
\item $Y_t = O(f_t)$ if $\{ Y_t / f_t \}$ is bounded almost surely.
\item $Y_t = \OM(f_t)$ if $Y_t$ is positive almost surely, and
$\{ f_t/Y_t \}$ is bounded almost surely.
\item $Y_t = \Th(f_t)$ if $Y_t$ is both $O(f_t)$ and $\OM(f_t)$.
\item $Y_t = o(f_t)$ if $Y_t /f_t \ap 0$ almost surely as $\tai$.
\een
\end{definition}

With this definition, the following theorem holds; it is Theorem 5.2 of
\cite{MV-RLK-SGD-arxiv23,MV-RLK-SGD-JOTA24}.
Similar results can be found in \cite{Liu-Yuan-arxiv22}.

\begin{theorem}\label{thm:42}
Suppose $\{ V_t \} , \{ f_t \} , \{ g_t \} , \{ \al_t \}$ are
stochastic processes defined on some probability space $(\OM,\SI,P)$,
taking values in $[0,\infty)$, adapted to some
filtration $\{ \F_t \}$.
Suppose further that
\be\label{eq:426}
E_t(V_{t+1} ) \leq (1 + f_t) V_t + g_t - \al_t z_t \fa t ,
\ee
and in addition, almost surely
\bd
\sum_{t=0}^\infty f_t(\om) < \infty , \quad
\sum_{t=0}^\infty g_t(\om) < \infty , \quad
\sum_{t=0}^\infty \al_t(\om) = \infty .
\ed
Then $V_t = o(t^{-\la})$ for every $\la \in (0,1]$ such that
there exists a finite $T > 0$ such that
\be\label{eq:427}
\al_t(\om) - \la t^{-1} \geq 0 \fa t \geq T ,
\ee
and in addition
\be\label{eq:428}
\sum_{t=0}^\infty (t+1)^\la g_t(\om) < \infty , \quad
\sum_{t=0}^\infty [ \al_t(\om) - \la t^{-1} ] = \infty .
\ee
\end{theorem}

\subsection{Proof of Theorem \ref{thm:21}}\label{ssec:43}

The proof of Theorem \ref{thm:21} is based on applying the Robbins-Siegmund
theorem stated here as Lemma \ref{lemma:41} to the ``Lyapunov function''
\be\label{eq:431}
V_t := J(\ubold_t) + \nmeusq{\vbold_t} .
\ee
The reason for calling it a ``Lyapunov function''
is that its conditional expectation obeys the conditions of the
Robbins-Siegmund theorem; this in turn allows us to deduce the convergence
of $V_t$ to zero almost surely.
We will find an upper bound for $E_t(V_t)$ in the form
\be\label{eq:432}
E_t(V_t) \leq V_t + f_t V_t + g_t - \frac{1-\mub^2}{2}
- \frac{\al_t \bb}{2} \nmeusq{\gJu} - F_t ,
\ee
where the sequences $\{ f_t \}, \{ g_t \}$ are nonnegative and summable,
and $F_t$ is a quadratic form in $\gJu$ and $\nmeu{\vbold_t}$
which is positive definite \textit{for sufficiently large} $t$,
say for all $t \geq T$.
(In case these entities are random, these conditions hold almost surely),
Since we can always start our analysis at time $T$, we can neglect the term
$-F_t$ for all $t \geq T$, and apply Lemma \ref{lemma:41}.

Going forward, we will avoid a lot of cumbersome notation if we agree to
refer to a nonnegative sequence $\{ z_t \}$ as a Well-Behaved Function ({\tt WBF})
if there exist \textit{nonnegative summable} sequences $\{ f_t \} , \{ g_t \}$
such that
\be\label{eq:433}
z_t \leq g_t + f_t V_t , \fa t \geq 0 .
\ee
In case the various entities are random, the assumptions (inequality and
summability) hold almost surely.
Clearly the sum of two {\tt WBF} is again a {\tt WBF}, and a {\tt WBF} multiplied by
a bounded sequence is again a {\tt WBF}.
Therefore any {\tt WBF} can be absorbed into the terms $g_t + f_t V_t$,
and it is not necessary to keep careful track of them.

Bounding $E_t(V_{t+1})$ involves several intricate computations.
For this purpose, it is now shown that the first two conditions in
\eqref{eq:221} imply that
\be\label{eq:434}
\sum_{t=0}^\infty \al_t^2 B_t < \infty , \quad
\sum_{t=0}^\infty \al_t^2 B_t^2 < \infty .
\ee
The proof of this claim is as follows:
The first bound in \eqref{eq:221} implies in particular that $\al_t \ap 0$
as $\tai$, and hence $\al_t$ is bounded, say by $\bar{\al}$.
Therefore
\bd
\sum_{t=0}^\infty \al_t^2 B_t \leq \bar{\al} 
\sum_{t=0}^\infty \al_t B_t < \infty .
\ed
This is the first bound in \eqref{eq:434}.
As for the second bound, recall that 
every (absolutely) summable sequence is also square summable.
Therefore we can append the two bounds in \eqref{eq:434}
to the three bounds in \eqref{eq:221}.

The first step in proceeding further is to reformulate the bounds 
\eqref{eq:219} and \eqref{eq:2110}, which are stated in terms of
$\gJw$, in terms of $\gJu$.
Accordingly, we modify \eqref{eq:217} by defining
\be\label{eq:435}
\xbb_t = \z_t - \gJu = E_t( \h_{t+1}) - \gJu .
\ee
The objectives are to find bounds for $\nmeusq{\xbb_t}$ and
$E_t( \nmeusq{\bzt} )$ in terms of $V_t$.
Throughout we use the bound \eqref{eq:211}, namely
$\nmeusq{\gJu} \leq 2L J(\ubold_t)$.
We also make repeated use of the obvious inequalities
\be\label{eq:435a}
x \leq (1+x^2)/2 , xy \leq (x^2 + y^2 )/2, \fa x, y \in \R .
\ee

We begin with a bound for $\nmeu{\xbb_t}$.
Observe that
\bd
\xbb_t  = \z_t - \gJu = \x_t + \gJw - \gJu
\ed
Hence it follows from \eqref{eq:414} and \eqref{eq:219} that
\be\label{eq:436}
\begin{split}
\nmeu{\xbb_t} &\leq \nmeu{\x_t} + L \nmeu{ \w_t - \ubold_t } \\
&\leq B_t( 1 + \nmeu{\gJw} ) + L k_t \nmeu{\vbold_t} \\
&\leq B_t( 1 + \nmeu{\gJu} + L k_t \nmeu{\vbold_t} ) + L k_t \nmeu{\vbold_t} \\
&\leq B_t( 1 + \nmeu{\gJu} + L \kb \nmeu{\vbold_t} ) + L \kb \nmeu{\vbold_t} .
\end{split}
\ee

Next, we can find a bound for $\nmeusq{\xbb_t}$ starting
from \eqref{eq:436}, and arrive at
\be\label{eq:437}
\begin{split}
\nmeusq{\xbb_t} & \leq B_t^2 ( 1 + \nmeu{\gJu} + L \kb \nmeu{\vbold_t} )^2 \\
& +  2 B_t L \kb ( 1 + \nmeu{\gJu} + L \kb \nmeu{\vbold_t} ) \cdot 
\nmeu{\vbold_t} 
+ (L \kb)^2 \nmeusq{\vbold_t} .
\end{split}
\ee
Note that terms of the form $\nmeu{\gJu}$, $\nmeu{\vbold_t}$, $\nmeusq{\gJu}$,
$\nmeu{\gJu} \cdot \nmeu{\vbold_t}$ and $\nmeusq{\vbold_t}$
can be bounded by terms of the form $C_1 + C_2 V_t$ for suitable constants
$C_1$ and $C_2$.
Clearly $\nmeusq{\vbold_t} \leq V_t$.
The rest can be bounded repeatedly using \eqref{eq:435a}.
Specifically
\bd
\nmeu{\gJu} \leq \half (1 + \nmeusq{\gJu} )
\leq \half + L J(\ubold_t) \leq \half + L V_t  ,
\ed
\bd
\nmeu{\vbold_t} \leq \half (1 + \nmeusq{\vbold_t} )
\leq \half (1 + V_t) ,
\ed
\bd
\nmeusq{\gJu} \leq 2L J(\ubold_t) \leq 2L V_t ,
\ed
\bd
\begin{split}
\nmeu{\gJu} \cdot \nmeu{\vbold_t}
& \leq \half ( \nmeusq{\gJu} + \nmeusq{\vbold_t} ) \\
& \leq L J(\ubold_t) + \half \nmeusq{\vbold_t} ) 
\leq \max\{L, 1/2 \} V_t .
\end{split}
\ed
Applying all these bounds to \eqref{eq:436} shows that
\be\label{eq:438}
\nmeusq{\xbb_t} \leq B_t^2 ( D_{11} + D_{12} V_t ) +
B_t ( D_{21} + D_{22} V_t )
+ (L \kb)^2 \nmeusq{\vbold_t} ,
\ee
for suitable constants $D_{11}$ through $D_{22}$.

For future use, we also bound $\nmeusq{\z_t}$.
Since $\z_t = \xbb_t + \gJu$, we can write
\be\label{eq:439}
\begin{split}
\nmeusq{\z_t} & \leq \nmeusq{\xbb_t} + 2 \nmeu{\xbb_t} \cdot \nmeu{\gJu}
+ \nmeusq{\gJu} \\
& \leq \nmeusq{\xbb_t} + [ \nmeusq{\xbb_t} + \nmeusq{ \gJu } ]
+ \nmeusq{\gJu} \\
& = 2 \nmeusq{\xbb_t} + 4 L J(\ubold_t) .
\end{split}
\ee
Since $J(\ubold_t) \leq V_t$, we can substitute from \eqref{eq:438}
into \eqref{eq:439} to obtain the bound
\be\label{eq:4310}
\begin{split}
\nmeusq{\z_t} & \leq
B_t^2 ( D_{11} + D_{12} V_t ) + B_t ( D_{21} + D_{22} V_t ) + 4 L J(\ubold_t)
+ (L \kb)^2 \nmeusq{\vbold_t} \\
& \leq B_t^2 ( D_{11} + D_{12} V_t ) + B_t ( D_{21} + D_{22} V_t ) 
+ \max \{ 4L , (L\kb)^2 \} V_t .
\end{split}
\ee

With these bounds in place, we now proceed to prove \eqref{eq:431}.
Clearly
\be\label{eq:4311}
E_t(V_{t+1}) = E_t ( J(\ubold_{t+1}) ) + E_t ( \nmeusq{\vbold_{t+1}} ) .
\ee
So we bound each of these two terms individually.
First, it follows from \eqref{eq:4110} that
\bd
\nmeusq{\vbold_{t+1}} = \nmeusq{ \mu_t \vbold_t - \al_t \h_{t+1} } \\
 = \mu_t^2 \nmeusq{\vbold_t} - 2 \al_t \mu_t \IP{ \vbold_t }{\h_{t+1}}
+ \al_t^2 \nmeusq{\h_{t+1}} .
\ed
Therefore, from \eqref{eq:218}, we get
\be\label{eq:4312}
E_t( \nmeusq{\vbold_{t+1}} ) = \mu_t^2 \nmeusq{\vbold_t}
- 2 \al_t \mu_t \IP{ \vbold_t }{\z_t} 
+ \al_t^2 [ \nmeusq{\z_t} + E_t( \nmeusq{ \bzt } ) ] .
\ee
We can estimate the last two terms separately.

First,
\bd
- 2 \al_t \mu_t \IP{ \vbold_t }{\z_t} \leq 2 \al_t \mu_t \nmeu{\vbold_t} 
\cdot \nmeu{\z_t} \leq 2 \al_t \mub \nmeu{\vbold_t} \cdot \nmeu{\z_t} .
\ed
Note that the bound is unaffected by the presence or the absence of
the minus sign in front of the inner product.
Further,
\bd
2 \al_t \mub \nmeu{\vbold_t} \cdot \nmeu{\z_t}
\leq 2 \al_t \mub \nmeu{\vbold_t} \cdot [ \nmeu{\xbb_t} + \nmeu{ \gJu } ].
\ed
Substituting for $\nmeu{\xbb_t}$ from \eqref{eq:436}, and recalling that
$\{ \al_t B_t \}$ is a summable sequence leads to the observation that
\be\label{eq:4313}
2 \al_t \mub \nmeu{\vbold_t} \cdot \nmeu{\xbb_t} = {\tt WBF} + 2 \al_t
\mub L \kb \nmeusq{\vbold_t} + 2 \al_t \mub \nmeu{\vbold_t} \cdot \nmeu{\gJu} ,
\ee
where ``{\tt WBF}'' denotes a well-behaved function, defined in \eqref{eq:433}.
Therefore it is not necessary to write it out in detail.
As a result
\be\label{eq:4313a}
2 \al_t \mub \nmeu{\vbold_t} \cdot \nmeu{\z_t} = {\tt WBF} + 2 \al_t
\mub L \kb \nmeusq{\vbold_t} + 4 \al_t \mub \nmeu{\vbold_t} \cdot \nmeu{\gJu} ,
\ee

Next we bound the last term.
\be\label{eq:4314}
\al_t^2 [ \nmeusq{\z_t} + E_t( \nmeusq{ \bzt } ) ] 
= \al_t^2 \nmeusq{\z_t} + \al_t^2 E_t( \nmeusq{ \bzt } ) ] .
\ee
We already have a bound for $\nmeusq{\z_t}$, namely \eqref{eq:4310}.
As discussed earlier, the hypothesis \eqref{eq:221} implies \eqref{eq:434}.
Therefore the term $\al_t^2 \nmeusq{\z_t}$ is a {\tt WBF}.
So let us focus on $E_t( \nmeusq{ \bzt } ) ]$.
There is a bound on this quantity in \eqref{eq:2110}, but it is stated
in terms $J(\w_t)$.
The bound is now restated in terms of $J(\ubold_t)$, using Lemma 
\ref{lemma:42}.
We know from \eqref{eq:414} that $\ubold_t = \w_t + k_t \vbold_t$.
So applying Lemma \ref{lemma:42} gives
\be\label{eq:4315}
J(\w_t) = J(\ubold_t - k_t \vbold_t) 
\leq J(\ubold_t) - k_t \IP{ \gJu }{\vbold_t} + \frac{L k_t^2}{2}
\nmeusq{\vbold_t} .
\ee
Now Schwarz' inequality and \eqref{eq:435a} lead to
\be\label{eq:4316}
\begin{split}
- k_t \IP{ \gJu }{\vbold_t} & \leq \frac{k_t}{2}
[ \nmeusq{\gJu} + \nmeusq{ \vbold_t} ] \\
& \leq \frac{k_t}{2} [ 2L J(\ubold_t) + \nmeusq{ \vbold_t} ] 
\leq \frac{\kb}{2} [ 2L J(\ubold_t) + \nmeusq{ \vbold_t} ] ,
\end{split}
\ee
This can be substituted into \eqref{eq:4315} to give
\be\label{eq:4316a}
J(\w_t) \leq J(\ubold_t) + L \kb J(\ubold_t) + \left[
\frac{\kb}{2} + \frac{L \kb^2}{2} \right] \nmeusq{\vbold_t}
\leq D_3 V_t ,
\ee
where
\bd
D_3 = \max \left\{ L \kb , \frac{\kb}{2} + \frac{L \kb^2}{2} \right \} .
\ed
Therefore the bound in \eqref{eq:2110} can be reformulated as
\be\label{eq:4317}
\al_t^2 E_t( \nmeusq{ \bzt } ) ] \leq \al_t^2 D_3 V_t ,
\ee
which is a {\tt WBF} in view of the assumptions \eqref{eq:221}.
Substituting all these bounds into \eqref{eq:4312} gives
\be\label{eq:4318}
E_t( \nmeusq{ \vbold_{t+1} } ) \leq {\tt WBF} + \mu_t^2 \nmeusq{\vbold_t} + 2 \al_t
\mub L \kb \nmeusq{\vbold_t} + 2 \mub \al_t \nmeu{\vbold_t} \cdot \nmeu{\gJu} .
\ee

Next we turn our attention to $E_t( J(\ubold_{t+1}) )$.
Recall from \eqref{eq:419} that
\bd
\ubold_{t+1} =
\ubold_t + \de_t \mu_t \vbold_t - (b_t + k_{t+1}) \al_t \h_{t+1} .
\ed
Therefore, by applying Lemma \ref{lemma:42}, we get
\be\label{eq:4319}
\begin{split}
J(\ubold_{t+1}) & =
J( \ubold_t + \de_t \mu_t \vbold_t - (b_t + k_{t+1}) \al_t \h_{t+1} ) \\
& \leq J(\ubold_t) + \de_t \mu_t \IP{\gJu}{\vbold_t}
- \al_t (b_t + k_{t+1}) \IP{ \gJu}{\h_{t+1}} \\
& + \frac{L}{2} \nmeusq{ \de_t \mu_t \vbold_t - \al_t (b_t + k_{t+1}) \h_{t+1}}.\end{split}
\ee
From \eqref{eq:217} we have that
\bd
\nmeusq{ \de_t \mu_t \vbold_t - \al_t (b_t + k_{t+1}) \h_{t+1}}
= \nmeusq{ \de_t \mu_t \vbold_t - \al_t (b_t + k_{t+1}) \z_t
- \al_t (b_t + k_{t+1}) \bzt } .
\ed
However, since $E_t(\bzt) = \bz$, it follows that
\bd
\begin{split}
E_t( \nmeusq{ \de_t \mu_t \vbold_t - \al_t (b_t + k_{t+1}) \h_{t+1}} )
& = \nmeusq{ \de_t \mu_t \vbold_t - \al_t (b_t + k_{t+1}) \z_t } \\
& + \al_t^2 (b_t + k_{t+1})^2 E_t( \nmeusq{\bzt} ) .
\end{split}
\ed
Applying $E_t(\cdot)$ to both sides of \eqref{eq:4319}, and
substituting the above relationship, gives
\be\label{eq:4320}
\begin{split}
E_t( J(\ubold_{t+1}) ) & \leq J(\ubold_t) + \de_t \mu_t \IP{\gJu}{\vbold_t}
- \al_t (b_t + k_{t+1}) \IP{ \gJu}{\z_t} \\
& + \frac{L}{2} \nmeusq{ \de_t \mu_t \vbold_t - \al_t (b_t + k_{t+1}) \z_t }
+ \frac{L}{2} \al_t^2 (b_t + k_{t+1})^2 E_t( \nmeusq{\bzt} ) .
\end{split}
\ee

Now we analyze each of the terms in \eqref{eq:4320} individually.
Before doing so, we replace several functions of $t$ by their bounds.
Specifically
\bit
\item $\de_t$ could be positive or negative, but is assumed to converge to $0$.
Therefore $|\de_t|$ is bounded, say by $\delb$.
\item $\mu_t \in [0,\mub]$ where $\mub < 1$.
\item $b_t \in [\bul , \bb]$ where $0 < \bul \leq \bb$, and $k_t \in [0,\kb]$.
Therefore $b_t + k_{t+1} \in [\bul , \bb + \kb ]$.
\eit
With these observations, we have the following bounds:
\be\label{eq:4321}
\de_t \mu_t \IP{\gJu}{\vbold_t} \leq \delb \mub \nmeu{\gJu} \cdot 
\nmeu{\vbold_t} .
\ee
Next
\be\label{eq:4322}
\begin{split}
- \al_t (b_t + k_{t+1}) \IP{ \gJu}{\z_t}
& = - \al_t (b_t + k_{t+1}) \nmeusq{\gJu} \\
& - \al_t (b_t + k_{t+1}) \IP{ \gJu}{\xbb_t} \\
& \leq - \al_t \bul \nmeusq{\gJu} \\
& + \al_t (\bb + \kb) \nmeu{\gJu}
\cdot \nmeu{\xbb_t} .
\end{split}
\ee
To bound the last term on the right side of \eqref{eq:4322}, 
we use the bound on $\nmeu{\xbb_t}$ from \eqref{eq:436}, and
the summability of $\{ \al_t B_t \}$.
This gives
\be\label{eq:4323}
\al_t (\bb + \kb) \nmeu{\gJu} \cdot \nmeu{\xbb_t} 
\leq \al_t L \kb(\bb+\kb) \nmeu{\gJu} \cdot \nmeu{\vbold_t} + {\tt WBF} .
\ee
Next we tackle the first quadratic term on the right side of \eqref{eq:4320}.
\be\label{eq:4324}
\begin{split}
\frac{L}{2} \nmeusq{ \de_t \mu_t \vbold_t - \al_t (b_t + k_{t+1}) \z_t }
& = \frac{L}{2} \nmeusq{\de_t \mu_t \vbold_t} \\
& + \frac{\al_t^2 L}{2} (b_t + k_{t+1})^2 \nmeusq{\z_t} \\
& - \al_t L (b_t + k_{t+1}) \de_t \mu_t \IP{\vbold_t}{\z_t} .
\end{split}
\ee
Each of the three terms can be analyzed individually.
\be\label{eq:4324a}
\frac{L}{2} \nmeusq{\de_t \mu_t \vbold_t}
\leq \frac{L \mub^2 \de_t^2}{2} \nmeusq{\vbold_t} .
\ee
Next, from \eqref{eq:4310}, it follows that
\be\label{eq:4325}
\frac{\al_t^2 L}{2} (b_t + k_{t+1})^2 \nmeusq{\z_t} = {\tt WBF} .
\ee
Finally, we already have a bound for the cross-product term 
$\al_t \mu_t \IP{\vbold_t}{\z_t}$ from \eqref{eq:4313a}.
After multiplying this bound by $L \delb (\bb+\kb)$, we get
\be\label{eq:4326}
\begin{split}
\al_t L (b_t + k_{t+1}) \de_t \mu_t \IP{\vbold_t}{\z_t}
& \leq {\tt WBF} + \al_t L^2 \kb \delb ( \bb + \kb ) \nmeusq{\vbold_t} \\
& + 2 \al_t L \delb (\bb + \kb) \nmeu{\vbold_t} \cdot \nmeu{\gJu} .
\end{split}
\ee

Now we can add up all these bounds.
This gives
\be\label{eq:4327}
\begin{split}
E(V_{t+1}) & = E_t( J(\ubold_{t+1}) ) + E_t ( \nmeusq{\vbold_{t+1}} ) \\
& \leq J(\ubold_t) + \nmeusq{\vbold_t} - (1 - \mub^2) \nmeusq{\vbold_t}
- \al_t \bul \nmeusq{\gJu} \\
& + C_1 \al_t \nmeusq{\vbold_t} + C_2 \al_t \nmeu{\vbold_t} \cdot \nmeu{\gJu}
+ {\tt WBF} ,
\end{split}
\ee
where $C_1, C_2$ are some positive
constants whose precise value is not important.
Next, we can ``borrow'' half of each of the two negative terms in the above,
and rewrite the bound as
\be\label{eq:4328}
E(V_{t+1}) 
\leq V_t - \frac{1 - \mub^2}{2} \nmeusq{\vbold_t}
- \al_t \frac{\bul}{2} \nmeusq{\gJu} 
- F_t + {\tt WBF} ,
\ee
where $F_t$ is the quadratic form
\bd
F_t = [ \ba{cc} \nmeu{\vbold_t} & \nmeu{\gJu} \ea ] 
\left[ \ba{cc} \frac{1 - \mub^2}{2} - \al_t C_1 & -\al_t (C_2/2) \\
-\al_t (C_2/2) & \al_t (\bul/2) \ea \right]
\left[ \ba{c} \nmeu{\vbold_t} \\ \nmeu{\gJu} \ea \right] .
\ed

Let us define
\be\label{eq:4329}
K_t = \left[ \ba{cc} \frac{1 - \mub^2}{2} - \al_t C_1 & -\al_t (C_2/2) \\
-\al_t (C_2/2) & \al_t (\bul/2) \ea \right] .
\ee
It is now shown that $K_t$ is a positive definite matrix, and
$F_t$ is a positive definite form, for $t$ sufficiently
large; specifically, there exists a $T < \infty$ such that
$F_t \geq 0$ for all $t \geq $T.
Suppose we succeed in proving this.
Since we can always start our analysis of \eqref{eq:4317} starting at
time $T$, we can write
\be\label{eq:4329a}
E(V_{t+1}) \leq V_{t+1}
- \left(\frac{1-\mub^2}{2} \right) \nmeusq{\vbold_t} - \al_t \nmeusq{\gJw}
+ {\tt WBF} , \fa t \geq T .
\ee
In other words, the term $-F_t$ is gone.
Now \eqref{eq:4329} is in a form to which the Robbins-Siegmund theorem
(Lemma \ref{lemma:41}) can be applied.
So let us now establish the positive definiteness of the quadratic form
for sufficiently large $t$.
Note that a symmetric $2 \times 2$ matrix is positive definite if its trace
and its determinant are both positive.
In this case
\bd
\mathrm{tr}(K_t) = \frac{1 - \mub^2}{2} - (C_1 - (\bul/2)) \al_t , \quad
\mathrm{det}(K_t) = \frac{1-\mub^2}{2} \frac{\bul}{2} \al_t -
C_3 \al_t^2 ,
\ed
where $C_3$ is another constant.
Since, by hypothesis, $\sum_{t=0}^\infty \al_t^2 < \infty$,
it follows that $\al_t \ap 0$ as $\tai$.
Hence the trace of $K_t$ is positive for sufficiently large $t$.
Similarly, in the expression for the determinant of $K_t$, the positive
term is linear in $\al_t$, whereas the negative term is quadratic in $\al_t$.
Hence the determinant of $K_t$ is also positive for sufficiently large $t$.
Hence we conclude that $K_t$ is a positive definite matrix for sufficiently
large $t$.

With this observation, we can apply Theorem \ref{thm:41} to \eqref{eq:4329}.

We begin wih Item 1.
Note that all statements hold ``almost surely,'' so this qualifier is not
repeated each time.
Suppose \eqref{eq:425} holds.
Then the following conclusions follow from Theorem \ref{thm:41}:
\bit
\item $J(\ubold_t) + \nmeusq{\vbold_t}$ is bounded.
Moreover, there is a random variable $W$ such that
$J(\ubold_t) + \nmeusq{\vbold_t} \ap W$ (almost surely) as $\tai$.
\item Further, almost surely
\be\label{eq:43210}
\sum_{t=0}^\infty \left( \frac{1-\mub^2}{2} \right)
\nmeusq{\vbold_t} + \al_t \nmeusq{\gJu} < \infty .
\ee
\eit
Since the summands in \eqref{eq:43210} are both nonnegative, and
$(1 - \mub^2)/2$ is just a constant, it follows that
\be\label{eq:43211}
\sum_{t=0}^\infty \nmeusq{\vbold_t} < \infty ,
\ee
\be\label{eq:43212}
\sum_{t=0}^\infty \al_t \nmeusq{\gJu} < \infty .
\ee
Now \eqref{eq:43211} implies that $\nmeusq{\vbold_t} \ap 0$ as $\tai$,
i.e., that $\vbold_t \ap \bz$ as $\tai$.
In turn, if $J(\ubold_t) + \nmeusq{\vbold_t} \ap X$, then $J(\w_t) \ap X$ as $\tai$.

Now recall from \eqref{eq:123} that $\bth_t = \w_t - \e_t \vbold_t$.
Since $J(\cdot)$ is continuous and $\vbold_t \ap \bz$,
it follows that $J(\bth_t) \ap W$ as $\tai$.
The boundedness of $\{ J(\bth_t) \}$ follows from it being a convergent
sequence.
Finally, the boundedness of $\{ \gJt \}$ follows from Lemma \ref{lemma:21}.
Thus we have established Item 1.

Next we address Item 2.
Suppose \eqref{eq:222} holds.
Then it readily follows from \eqref{eq:43212} that
\bd
\liminf_{\tai} \nmeusq{\gJu} = 0 .
\ed
To translate this conclusion into the behavior of $\gJt$, we proceed
as follows:
It follows from the definitions of $\w_t$ and $\ubold_t$ that
\bd
\bth_t = \ubold_t - (k_t + \e_t) \vbold_t .
\ed
Since $\e_t$ and $k_t$ are bounded, 
$\vbold_t \ap \bz$ as $\tai$, and $\gJ(\cdot)$ is Lipschitz-continuous,
it can be concluded that
\bd
\liminf_{\tai} \nmeusq{\gJt} = 0 .
\ed
This is Item 2.

Next we address Item 3 of the theorem.
The hypotheses are that, in addition to \eqref{eq:221}, \eqref{eq:222}
also holds, and $J(\cdot)$ satisfies Property (KL').
Then by definition there exists
a function $\psi : \R \ap \R$ in Class $\B$ such that
$\nmeu{\gJt} \geq \psi(J(\bth_t))$.
Recall that all the stochastic processes are defined on some underlying
probability space $(\OM,\SI,P)$.
Define
\bd
\OM_0 := \{ \om \in \OM : J(\bth(\om)) \ap W(\om)
\& \nmeusq{\vbold_t(\om)} \ap 0 \} ,
\ed
\bd
\OM_1 := \{ \om \in \OM : \sum_{t=0}^\infty \al_t(\om) = \infty \} .
\ed
Note that if the step sizes are deterministic, then $\OM_1 = \OM$.
Define $\OM_2 = \OM_0 \cap \OM_1$, and note that $P(\OM_2) = 1$,
by Item 1.

The objective is to show that $W(\om) = 0$ for all $\om \in \OM_2$.
Once this is done, it would follow from Lemma \ref{lemma:41} that
\bd
\nmeu{\gJ(\bth_t(\om))} \leq [2L J (\bth_t(\om))]^{1/2}
\ap 0 \mbox{ as } \tai , \fa \om \in \OM_2 .
\ed
Accordingly, suppose that, for some $\om \in \OM_0$, we have that $W(\om) > 0$,
say $W(\om) = 2p$.
Define
\bd
G(\om) := \sup_t J(\bth_t(\om)) .
\ed
Then $G(\om) < \infty$ because $\{ J(\bth_t(\om)) \}$ is a convergent sequence.
Define
\bd
q := \half \inf_{p \leq r \leq M(\om)} \psi(r) .
\ed
Then $q> 0$ because $\psi(\cdot)$ is a function of Class $\B$.
Now choose a $T_0 < \infty$ such that $J(\bth(\om)) \geq p$
for all $t \geq T_0$.
By the (KL') property, it follows that
\bd
\nmeu{\gJ(\bth(\om))} \geq 2 q , \fa t \geq T_0 .
\ed
Next, choose $T_1 < \infty$ such that $\nmeu{\vbold_t(\om)} \leq q/L$
for all $t \geq T_1$, and define $T_2 = \min \{ T_0, T_1 \}$.
Then it follows from the Lipschitz continuity of $\gJ(\cdot)$ that
\be\label{eq:43213}
\nmeu{\gJ(\w_t(\om))} \geq \nmeu{\gJ(\bth_t(\om))} - L \nmeu{\vbold_t(\om)}
\geq q, \fa t \geq T_2 .
\ee
On the other hand, because $\om \in \OM_2$, we have that
\be\label{eq:43214}
\sum_{t=T_2} \al_t(\om) = \infty .
\ee
Thus \eqref{eq:43213} and \eqref{eq:43214} together imply that
\bd
\sum_{t=T_2}^\infty \al_t \nmeusq{ \gJt } = \infty .
\ed
Since this contradicts \eqref{eq:43212}, we conclude that no such
$\om \in \OM_2$ can exist.
In other words $W(\om) = 0$ for all $\om \in \OM_2$.
This establishes Item 3.

Item 4 is a ready consequence of Item 3 and Property (NSC).
If $\{ J(\bth_t) \}$ is bounded, then the fact that $J(\cdot)$
has compact level sets means that $\{ \bth_t \}$ is bounded.
Then the fact that $J(\bth_t) \ap 0$ as $\tai$ implies that $\rho(\bth_t) \ap 0$
as $\tai$; in other words, the distance from the iterate $\bth_t$ to
the set $S_J$ of global minima approaches zero.
Note that it is \textit{not} assumed that $S_J$ consists of a singleton.

This completes the proof of Theorem \ref{thm:21}.

\subsection{Proof of Theorem \ref{thm:22}}\label{ssec:44}

The proof, based on Theorem \ref{thm:21}, is basically the same as that
of Theorem 6.2 of \cite{MV-RLK-SGD-arxiv23,MV-RLK-SGD-JOTA24}.
The only difference is that the bound \eqref{eq:4329} holds only after
some time $T$.
Clearly this does not affect the \textit{asymptotic} rate of convergence.
Nevertheless, in the interests of completeness, the proof is
\textit{sketched} here.

The hypotheses on the various constants imply that
\bd
\al_t^2 = O(t^{-2+2 \phi}) , \al_t^2 M_t^2 = O(t^{-2+2 (\phi+\de)}) ,
\al_t B_t = O(t^{-1 + \phi - \g}) ,
\ed
while $\al_t^2 B_t$ and $\al_t^2 B_t^2$ decay faster than $\al_t B_t$.
Hence both $\{ f_t \}$ and $\{ g_t \}$ are summable if
\bd
-2 + 2 \phi < -1 , -2 + 2(\phi + \de) < -1 , -1 + \phi - \g < -1 .
\ed
The three inequalities are satisfied if $\phi$ satisfies \eqref{eq:225}.
Next, let us define $\nu$ as in \eqref{eq:226}, and apply 
Theorem \ref{thm:21}.
This leads to the conclusion that
$J(\ubold_t) + \nmeusq{\vbold_t} = o(t^{-\la})$ for every $\la \in (0,\nu)$.
In turn this means that, individually, both $J(\ubold_t)$ and
$\nmeusq{\vbold_t}$ are $o(t^{-\la})$ for every $\la \in (0,\nu)$.
Since $\bth_t = \ubold_t - (k_t + \e_t) \vbold_t$, and both $\e_t$
and $k_t$ are bounded, this leads to $J(\bth_t) = o(t^{-\la})$
for every $\la \in (0,\nu)$.
Finally, the (PL) property leads to $\nmeusq{\gJt} = o(t^{-\la})$
for every $\la \in (0,\nu)$.
If we choose the step size sequence to decay very slowly, then the
bound in \eqref{eq:227} follows readily.
This completes the proof of Theorem \ref{thm:22}.

\section{Conclusions and Future Research}\label{sec:Conc}

In this paper, we have presented a ``unified'' momentum-based algorithm
for convex and nonconvex optimization, which includes both the Stochastic
Heavy Ball (SHB) and Stochastic Nesterov Accelerated Gradient (SNAG)
algorithms as special cases.
One novel feature of our algorithm is that the momentum parameter
$\mu_t$ in \eqref{eq:122} is permitted to vary with time (iteration counter).
We build on the ``transformation of variables'' approach in
\cite{MV-TUKR-arxiv25} to convert the updating equations into two
\textit{asymptotically decoupled} equations; note that when the momentum
parameter is constant, the transformed equations would be \textit{exactly}
decoupled, as in \cite{MV-TUKR-arxiv25}.
However, the scope of the present paper is broader, because only the SHB
algorithm is studied in \cite{MV-TUKR-arxiv25}.

We have also analyzed the behavior of the decoupling approach proposed in
\cite{Sebbouh-et-al-CoLT21} for the SHB algorithm, and showed that it is
not feasible when the momentum parameter varies with time.
The details can be found in the Appendix.

In this paper, even though the momentum parameter $\mu_t$ is allowed to
vary with time, it is assumed that it is bounded away from one.
However, in the standard version of SNAG, $\mu_t$
approaches one; hence the theory in this paper does not apply.
It is therefore worth pondering whether the present
approach can somehow be modified to cover the standard version of SNAG.

In this connection, it is worthwhile to recall a numerical example
from \cite{MV-TUKR-arxiv25}.
A variety of computational algorithms are appied in
\cite[Section 4]{MV-TUKR-arxiv25} to, among others, the following
function, referred to as $J_2(\cdot)$ there:
\bd
J(\bth) = \bth^\top A \bth + 3\sin^2 ( \IP{\oneb}{\bth} ) ,
\ed
where $\bth_t$ is a vector of 1 million parameters,
$A$ is a block-diagonal matrix of size ($10^6 \times 10^6$),
consisting of 100 Hilbert matrices, each of dimension $10^4 \times 10^4$,
and $\oneb$ denotes a column vector of all ones.
This function is \textit{not convex}, but it \textit{does} satisfy
the PL property, as can be easily verified.

In \cite{MV-TUKR-arxiv25} the authors attempted to minimize this function using
approximate gradients computed using the SPSA approach of \eqref{eq:224a},
and a variety of optimization algorithms.
The results, extracted from \cite{MV-TUKR-arxiv25}, are shown in 
Figure \ref{fig:1}.
\bfig
\bc
\includegraphics[width=125mm]{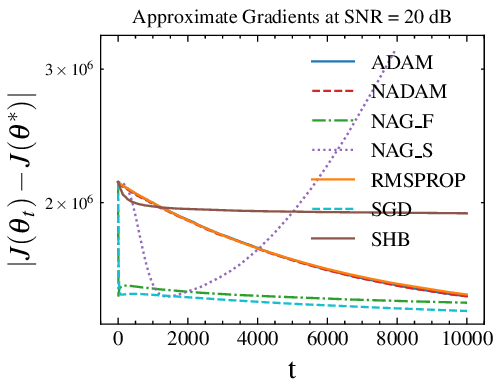}
\ec
\caption{Computational results for the numerical example with a variety
of algorithms}
\label{fig:1}
\efig
As can be seen from the figure, the standard Nesterov algorithm (SNAG)
diverges almost immediately.
This counter-example shows that, if the momentum parameter converges to 1,
\textit{and approximate derivatives of the form \eqref{eq:224a} are used},
then SNAG may not converge.

We surmise that the main reason for the divergence of SNAG when
approximate gradients are used is that the constant $M_t$ in \eqref{eq:2110}
grows without bound, as $c_t \ap 0$.
It is therefore worthwhile exploring whether an appropriate convergence
theory can be developed in the present framework, wherein
$\mu_t \ap 1^-$ as $\tai$, but $M_t$ is bounded.
This is a worthwhile problem for future research.

\section{Appendix}

In this appendix, we analyze the behavior of the solutions of
\eqref{eq:314} introduced in \cite{Sebbouh-et-al-CoLT21},
reproduced here for the convenience of the reader.
\bd
\la_{t+1} = \frac{\la_t}{\mu_t} - 1 ,
\eta_t = (1 + \la_{t+1}) \al_t
\ed
Recall from \eqref{eq:311} that the convergence conditions in
\cite{Sebbouh-et-al-CoLT21} involve the ``synthetic'' step size sequence
$\{ \eta_t \}$.
The objective of the Appendix is to show that this approach is not feasible.
Specifically, if $\{ \mu_t \}$ is a decreasing sequence, then $\la_t \ap
\infty$ as $\tai$.
Thus, even if the original step size sequence $\{ \al_t \}$ is
square-summable, the synthetic sequence of step sizes $\{ \eta_t \}$
need not be.
Thus the assumptions in \eqref{eq:311} are strictly stronger than 
the standard Robbins-Monro conditions, which is what is assumed here.
In the other direction, if $\{ \mu_t \}$ is an increasing sequence
bounded away from $1$, eventually $1 + \la_{t+1} < 0$, thus leading to
a negative step size $\eta_t$, which is absurd.
Thus the point is that, while the approach in
\cite{Sebbouh-et-al-CoLT21} is quite elegant, it is not practical.

We begin by presenting a ``closed-form'' formula for $\la_{t+1}$ as a function
of the $\mu_t$ sequence.
Write the first equation in \eqref{eq:314} as
\be\label{eq:318f}
\begin{split}
\la_{t+1} &= \frac{\la_t}{\mu_t} - 1 = \frac{\la_t - \la_{t-1}}{\mu_t}
+ \frac{\la_{t-1}}{\mu_t} - 1 \\
&= \frac{1}{\mu_t} ( \la_t - \la_{t-1} ) +
\left( \frac{1}{\mu_t} - \frac{1}{\mu_{t-1}} \right) \la_{t-1} 
+ \frac{\la_{t-1}}{\mu_{t-1}} - 1 \\
&= \la_t + \frac{1}{\mu_t} ( \la_t - \la_{t-1} )
+ \left( \frac{1}{\mu_t} - \frac{1}{\mu_{t-1}} \right) \la_{t-1} .
\end{split}
\ee
Therefore
\be\label{eq:319f}
\la_{t+1} - \la_t = \frac{1}{\mu_t} ( \la_t - \la_{t-1} )
+ \left( \frac{1}{\mu_t} - \frac{1}{\mu_{t-1}} \right) \la_{t-1} .
\ee
It is easy to show by induction
that a ``closed-form'' solution to \eqref{eq:319f} is
\be\label{eq:3110}
\la_{t+1} = \la_t + \left[ \prod_{\tau = 1}^t \frac{1}{\mu_\tau} \right]
( \la_1 - \la_0) 
+  \sum_{\tau = 1}^t \left[ \prod_{s = \tau}^{t-1} \frac{1}{\mu_s} \right]
\left( \frac{1}{\mu_\tau} - \frac{1}{\mu_{\tau-1} }\right) \la_{\tau-1} ,
\ee
where empty products are taken as $1$ and empty sums are taken as $0$.
Note that $\la_0$ is unspecified.
So if we take $\la_0 = \mu_0/(1 - \mu_0)$, then
\bd
\la_1 = \frac{\la_0}{\mu_0} - 1 = \frac{1}{1-\mu_0} - 1 
= \frac{\mu_0}{1-\mu_0} = \la_0 .
\ed
With this choice, the first term in \eqref{eq:3110} drops out;
but this is not much of a simplification.
Note also that if $\mu_t = \mu$ for all $t$, then $\la_t = \la_0$ for all $t$.

Now let us analyze the behavior of $\la_t$ in two specific situations:
(i) $\{ \mu_t \}$ is a strictly decreasing, i.e., $\mu_t < \mu_{t-1}$
for all $t$, and (ii) $\{ \mu_t \}$ is strictly increasing, but
bounded above by some $\mub < 1$.
In principle, the closed form solution \eqref{eq:3110} can be used to
analyze arbitrary sequences $\{ \mu_t \}$.
However, the two situations studied here are perhaps the most natural.

\begin{lemma}\label{lemma:A1}
Suppose $\la_0$ is chosen as $\mu_0/(1-\mu_0)$, so that $\la_1 = \la_0$.
Suppose further that $\mu_t < \mu_{t-1}$ for all $t \geq 1$.
Then $\la_t \ap \infty$ as $\tai$.
\end{lemma}

\begin{proof}
The first step is to show that $\la_{t+1} > \la_t$ for all $t \geq 1$.
The proof is by induction.
First, for $t=1$, we have that
\bd
\la_2 = \frac{\la_1}{\mu_1} - 1 > \frac{\la_1}{\mu_0} - 1
= \frac{1}{1 - \mu_0} - 1 = \la_1 .
\ed
Next suppose $\la_t > \la_{t-1}$.
Then
\bd
\la_{t+1} = \frac{\la_t}{\mu_t} - 1 > \frac{\la_{t-1}}{\mu_{t-1}} - 1 = \la_t .
\ed
This completes the proof by induction.

Next, we invoke the recursion \eqref{eq:319f}.
\bd
\la_{t+1} - \la_t = \frac{1}{\mu_t} ( \la_t - \la_{t-1} )
+ \left( \frac{1}{\mu_t} - \frac{1}{\mu_{t-1}} \right) \la_{t-1} .
\ed
The fact that $\mu_t < \mu_{t-1}$ implies that
\bd
\left( \frac{1}{\mu_t} - \frac{1}{\mu_{t-1}} \right) \la_{t-1} > 0 ,
\fa t \geq 2 .
\ed
Hence
\bd
\la_{t+1} - \la_t > \frac{1}{\mu_t} ( \la_t - \la_{t-1} )
> \frac{1}{\mu_2} ( \la_t - \la_{t-1} ) , \fa t \geq 2 .
\ed
As a consequence
we get
\bd
\la_{t+1} - \la_t > \left[ \prod_{s=2}^t \frac{1}{\mu_s}  \right]
( \la_2 - \la_1 )
> \left( \frac{1}{\mu_2} \right)^{t-1}
(\la_2 - \la_1) , \fa t \geq 2 .
\ed
We can add the above bound for all $t$.
Because it is a telescoping sum, we get
\bd
\la_{t+1} = \la_2 + \sum_{k=2}^t ( \la_{k+1} - \la_k )
\geq (\la_2 - \la_1) \sum_{k=2}^t \left( \frac{1}{\mu_2} \right)^{k-1} 
\ap \infty \mbox{ as } \tai .
\ed
\end{proof}

\begin{lemma}\label{lemma:A2}
Suppose $\mu_{t-1} < \mu_t < 1$ for all $t$, and that
\be\label{eq:3111}
\prod_{\tau = 2}^t \left( \frac{1}{\mu_\tau} \right) \ap \infty
\mbox{ as } \tai .
\ee
Then there exists a finite $t_0$ such that
\be\label{eq:3112}
1 + \la_t < 0 , \fa t \geq t_0 .
\ee
In particular, if $\mu_t \leq \mub < 1$ for all $t$, then we can take
\be\label{eq:3113}
t_0 = 3 + \log_{(1/\mub)} \frac{\la_1}{\la_1 - \la_2} .
\ee
\end{lemma}

\begin{proof}
Observe that
\bd
\la_2 = \frac{\la_1}{\mu_1} - 1 < \frac{\la_1}{\mu_0} - 1 = \la_1 .
\ed
Now suppose that $\la_t < \la_{t-1}$.
Then
\bd
\la_{t+1} = \frac{\la_t}{\mu_t} - 1 < \frac{\la_{t-1}}{\mu_{t-1}} - 1 = \la_t .
\ed
After observing that
\bd
\frac{1}{\mu_t} - \frac{1}{\mu_{t-1}} < 0 ,
\ed
we can rewrite\eqref{eq:319f} as
\be\label{eq:3114}
\la_t - \la_{t+1} = \frac{1}{\mu_t} ( \la_{t-1} - \la_t )
+ \left( \frac{1}{\mu_{t-1}} - \frac{1}{\mu_t} \right) \la_{t-1} .
\ee

Now suppose $\la_\tau > 0$ for $\tau = 1 , \cdots , t$.
Then \eqref{eq:3114} implies that
\be\label{eq:3115}
\la_\tau - \la_{\tau+1} > \frac{1}{\mu_\tau} ( \la_{\tau-1} - \la_\tau )
> \left( \prod_{k=2}^\tau \frac{1}{\mu_k} \right) ( \la_1 - \la_2 ) .
\ee
\be\label{eq:3116}
\la_1 - \la_{t+1} = \sum_{\tau = 1}^t ( \la_\tau - \la_{\tau-1} ) 
> \left[ \sum_{\tau=1}^t \left( \prod_{k=2}^t \frac{1}{\mu_k} \right)
\right] ( \la_1 - \la_2 ) .
\ee
Consequently
\be\label{eq:3117}
\begin{split}
\la_{t+1} & < \la_1 - \left[ \sum_{\tau=1}^t \left( \prod_{k=2}^\tau
\frac{1}{\mu_k} \right) \right] ( \la_1 - \la_2 ) \\
& < \la_1 - \left( \prod_{k=2}^t
\frac{1}{\mu_k} \right) ( \la_1 - \la_2 ) .
\end{split}
\ee
Now choose $T$ such that
\be\label{eq:3118}
\left( \prod_{k=2}^T \frac{1}{\mu_k} \right) > \frac{\la_1}{\la_1 - \la_2} .
\ee
This is possible in view of \eqref{eq:3111}.
Then there are two possibilities:
(i) $\la_\tau > 0$ for $\tau = 2 , \cdots , T$.
Then $\la_{T+1} < 0$ by virtue of \eqref{eq:3118}.
In this case we have that
\bd
\la_{T+2} = \frac{\la_{T+1}}{\mu_{T+1}} - 1 < -1 .
\ed
Therefore $\la_{T+2} + 1 < 0$.
The argument can be repeated, to show that $\la_t + 1 < 0$ for
all $t \geq T+2$.
Hence we can take $t_0 = T+2$ in \eqref{eq:3112}.
(ii) There exists a $\tau$ between $2$ and $T$ such that $\la_\tau \leq 0$.
By the above reasoning, it follows that
\bd
\la_{\tau+1} = \frac{\la_\tau}{\mu_\tau} - 1 \leq -1 < 0 .
\ed
Therefore
\bd
\la_{\tau+2} = \frac{\la_{\tau+1}}{\mu_{\tau+1}} - 1 < - 1 ,
\mbox{ or } \la_{\tau+2} + 1 < 0 .
\ed
As above, this leads to the conclusion that
$\la_t + 1 < 0$ for all $t \geq \tau + 2$.
Since $\tau \leq T$, we can conclude as before that $\la_t + 1 < 0$ for
all $t \geq T+2$.
Hence we can again take $t_0 = T+2$ in \eqref{eq:3112}.

To prove the last claim, suppose that $\mu_t \leq \mub < 1$ for all $t$.
Then we can replace \eqref{eq:3118} by
\bd
\left( \frac{1}{\mub} \right)^{T-1} \geq \frac{\la_1}{\la_1 - \la_2} .
\ed
Solving for $T$ and choosing $t_0 = T+2$ gives \eqref{eq:3113}.
\end{proof}


\end{document}